\documentclass[review]{elsarticle}

\RequirePackage{fix-cm}
\usepackage{amsthm}
\usepackage{amssymb,amsmath}
\usepackage{times}
\usepackage{graphicx}
\usepackage{epstopdf}
\usepackage{subfigure}
\usepackage{multirow}
\usepackage{tabularx}
\usepackage{color,soul}
\usepackage{algorithm}
\usepackage{algorithmicx}
\usepackage{algpseudocode}
\usepackage{url}
\usepackage{amssymb}
\usepackage{multirow}
\usepackage{array}
\usepackage{ltxtable}
\usepackage{tabularx}

\usepackage{booktabs}

\newtheorem{lemma}{Lemma}
\newtheorem{theorem}{Theorem}

\newcolumntype{L}{>{\raggedright\arraybackslash}X} 
\newcolumntype{R}{>{\raggedleft\arraybackslash}X}  
\DeclareMathOperator*{\argmin}{arg\,\min}

\newcommand{\prox}{\mathbf{prox}}

\biboptions{authoryear}

\begin{document}

\begin{frontmatter}

\title{Stochastic DCA for minimizing a large sum of DC functions with application to Multi-class Logistic Regression}

\author[address]{Hoai An Le Thi\corref{mycorrespondingauthor}}
\cortext[mycorrespondingauthor]{Corresponding author}
\ead{hoai-an.le-thi@univ-lorraine.fr}
\address[address]{Department of Computer Science and Applications, LGIPM, \\ University of Lorraine, France}

\author[address]{Hoai Minh Le}
\ead{minh.le@univ-lorraine.fr}

\author[address]{Duy Nhat Phan}
\ead{nhatsp@gmail.com}

\author[address]{Bach Tran}
\ead{bach.tran@univ-lorraine.fr}



%
%

\begin{abstract}
We consider the large sum of DC (Difference of Convex) functions minimization problem which appear in several different areas, especially in stochastic optimization and machine learning. 
Two DCA (DC Algorithm) based algorithms are proposed: stochastic DCA and inexact stochastic DCA. 
We prove that the convergence of both algorithms to a critical point is guaranteed with probability one. 
Furthermore, we develop our stochastic DCA for solving an important problem in multi-task learning, namely group variables selection in multi class logistic regression. 
The corresponding stochastic DCA is very inexpensive, all computations are explicit. 
Numerical experiments on several benchmark datasets and synthetic datasets illustrate the efficiency of our algorithms and their superiority over existing methods, with respect to classification accuracy, sparsity of solution as well as running time.
\end{abstract}


\end{frontmatter}


\section{Introduction}

We address the so called \textit{large sum of DC 
functions minimization} problem which takes the form 
\begin{equation}
\min\limits_{x\in \mathbb{R}^{d}}\left\{ F(x):=\frac{1}{n}%
\sum_{i=1}^{n}F_{i}(x)\right\} ,  \label{sum-model}
\end{equation}%
where $F_{i}$ are DC functions, i.e., $F_{i}(x)=g_{i}(x)-h_{i}(x)$ with $%
g_{i}$ being lower semi-continuous proper convex and $h_{i}$ being convex, and $%
n$ is a very large integer number. The problem of minimizing $F$ under a
convex set $\Omega$ is also of the type (\ref{sum-model}), as the convex
constraint $x\in \Omega$ can be incorporated into the objective function $F$
via the indicator function $\chi _{\Omega}$ on $\Omega$ defined by $\chi
_{\Omega}(x)=0$ if $x\in \Omega$, $+\infty $ otherwise.  Our study is motivated
by the fact that the problem (\ref{sum-model}) appears in several different
contexts, especially in stochastic optimization and machine learning. For
instance, let us consider the minimization of expected loss in stochastic
programming 
\begin{equation}
\min_{x\in \Omega }\mathbb{E}[f(x,\xi )],  \label{min-expected-loss}
\end{equation}%
where $f$ is a loss function of variables $x$ and $\xi $, and $\xi $ is a
random variable. A standard approach for solving (\ref%
{min-expected-loss-approx}) is the sample average method \citep{Healy1991} which
approximates the problem \eqref{min-expected-loss} by 
\begin{equation}
\min_{x\in \Omega }\frac{1}{n}\sum_{i=1}^{n}f(x,\xi _{i}),
\label{min-expected-loss-approx}
\end{equation}
where $\xi _{1},...,\xi _{n}$ are independent variables, identically
distributed realizations of $\xi $. When the loss function $f$ is DC, the
problem \eqref{min-expected-loss-approx} takes the form of \eqref{sum-model}
with $F_{i}(x)=f(x,\xi _{i})+\chi _{\Omega }(x)$. Obviously, the larger $n$
is, the better approximation will be. Hence, a good approximate model of the form
(\ref{min-expected-loss-approx})
  in average sample
methods 
requires an extremely
large number $n$.


Furthermore, let us consider an important problem in machine learning, the
multi-task learning. Let $T$ be the number of tasks. For the $j$-th task,
the training set $\mathcal{D}_{j}$ consists of $n_{j}$ labeled data points
in the form of ordered pairs $(x_{i}^{j},y_{i}^{j}),i=1,...,n_{j}$, with $%
x_{i}^{j}\in \mathbb{R}^{d}$ and its corresponding output $y_{i}^{j}\in 
\mathbb{R}$. Multi-task learning aims to estimate $T$ predictive functions $%
f_{\theta }^{j}(x):\mathbb{R}^{d}\rightarrow \mathbb{R}^{m},j=1,...,T$,
which fit well the data. The multi-task learning can be formulated as 
\begin{equation}
\min_{\theta }\left\{ \sum_{j=1}^{T}\sum_{i=1}^{n_{j}}\mathcal{L}%
(y_{i}^{j},f_{\theta }^{j}(x_{i}^{j}))+\lambda p(\theta )\right\} ,
\label{mul:task}
\end{equation}%
where $\mathcal{L}$ denotes the loss function, $p$ is a regularization term
and $\lambda >0$ is a trade-off parameter. For a good learning process, $\sum_{j=1}^Tn_{j}$ is, in general, a very large number. Clearly, this problem takes
the form of \eqref{sum-model} when $\mathcal{L}$ and $p$ are DC functions.
We observe that numerous loss functions in machine learning (e.g. least
square loss, squared hing loss, ramp loss, logistic loss, sigmoidal loss, etc) are DC. On
another hand, most of existing regularizations can be expressed as DC
functions. For instance, in learning with sparsity problems involving the
zero norm (which include, among of others, variable~/~group variable
selection in classification, sparse regression, compressed sensing) all
standard nonconvex regularizations studied in the literature are DC
functions \citep{LT15}. Moreover, in many applications dealing with big
data, the number of both variables and samples are very large.

The problem (\ref{sum-model}) has a double difficulties due to the
nonconvexity of $F_{i}$ and the large value of $n$. Meanwhile, the sum
structure of $F$ enjoys an advantage: one can work on $F_{i}$ instead of the
whole function $F$. Since all $F_{i}$ are DC functions, $F$ is DC too, 
and therefore (\ref{sum-model}) is a standard DC program, i.e., minimizing a
DC function under a convex set and/or the whole space.

To the best of our knowledge, although several methods have been developed
for solving different special cases of \eqref{sum-model}, there is no
existing work that considers the general problem (\ref{sum-model}) as well.
The stochastic gradient (SG) method was first introduced in \cite%
{Robbins51} and then developed in \cite{Bottou98,Lecun98}
for solving  \eqref{min-expected-loss-approx} in the unconstrained case ($%
\Omega = R^d$) with $f(\cdot,\xi_i)$ being smooth functions. The SG method
chooses $i_l\in\{1,...,n\}$ randomly and takes the update 
\begin{equation}
x^{l+1} = x^l - \alpha_l \nabla f(x^l,\xi_{i_l}),
\end{equation}
where $\alpha_l$ is the step size and $\nabla f(x^l,\xi_{i_l})$ is a
stochastic gradient. Later, \cite{Berinc2011,Bertsekas2010} proposed the proximal
stochastic subgradient methods (also referred as incremental proximal
methods) for solving \eqref{min-expected-loss-approx} in convex case, i.e., $%
\Omega$ is a closed convex set and $f(\cdot,\xi_i)$ are convex functions.
The computational cost per iteration of these basic SG methods is very
cheap, however, due to the variance introduced by random sampling, their
convergence rate are slower than the ''full'' gradient methods. Hence, some
SG methods for solving \eqref{min-expected-loss-approx} in unconstrained
differentiable convex case use either the average of the stored past
gradients or a multi-stage scheme to progressively reduce the variance of
the stochastic gradient (see e.g \cite%
{Schmidt2017,Shalev2013,Defazio2014a,Defazio2014b,Johnacc13}). With the variance reduction
techniques, other variants of the SG method have been proposed for nonconvex
problem \eqref{min-expected-loss-approx} where the $L$-smooth property is
required (see e.g. \cite{Maiinc15,Reddi2016,Allen-Zhu16}).

As (\ref{sum-model}) is a DC program, a natural way to tackle it is using
DCA (DC Algorithm) (see \citep{LeThi05,LeThi18,PLT98,PLT97,PLT14} and
references therein), an efficient approach in nonconvex programming
framework. DCA addresses the problem of minimizing a DC function on the
whole space $\mathbb{R}^{d}$ or on a closed convex set $\Omega \subset 
\mathbb{R}^{d}$. Generally speaking, a standard DC program takes the form: 
\begin{equation*}
\alpha =\inf \{F(x):=G(x)-H(x)\,|\,x\in \mathbb{R}^{d}\}\quad (P_{dc}),
\end{equation*}%
where $G,H$ are lower semi-continuous proper convex functions on $\mathbb{R}%
^{d}$. Such a function $F$ is called a DC function, and $G-H$ is a DC
decomposition of $F$ while $G$ and $H$ are the DC components of $F$. DCA has
been introduced in 1985~\cite{PhamDinh1986} and extensively developed since 1993 (%
\citep{LeThi05,LeThi18,PLT98,PLT97,PLT14} and references therein) to
become now classic and increasingly popular. Most of existing methods in
convex/nonconvex programming are special versions of DCA via appropriate DC
decompositions (see \citep{LeThi18}). In recent years, numerous DCA based
algorithms have been developed for successfully solving large-scale
nonsmooth/nonconvex programs appearing in several application areas,
especially in machine learning, communication system, biology, finance, etc. (see e.g. the list of references in \cite{LeThi2005,LeThi18}).
DCA has been proved to be a fast and scalable approach which is, thanks to
the effect of DC decompositions, more efficient than related methods. For a
comprehensible survey on thirty years of development of DCA, the reader is
referred to the recent paper \citep{LeThi18}. New trends in the
development of DCA concern novel versions of DCA  based algorithms (e.g.
online/stochastic/approximate/like DCA) to accelerate the convergence and to
deal with large-scale setting and big data. Our present work follows this
direction. 

The original key idea of DCA relies on the DC structure of the objective
function $F$. DCA consists in iteratively approximating the considered DC
program by a sequence of convex ones. More precisely, at each iteration $l$,
DCA approximates the second DC component $H(x)$ by its affine minorization $%
H_l(x):=H(x^l)+\langle x-x^l,y^l\rangle$, with $y^{l}\in \partial H(x^{l})$,
and minimizes the resulting convex function.

\noindent \textbf{Basic DCA scheme}

\noindent \textbf{Initialization:} Let $x^0 \in \text {dom }\partial H$, $l=0
$.

\noindent \textbf{For $l = 0, 1, \ldots$ until convergence of $\{x^l\}$:}

k1: Calculate $y^l \in \partial H (x^l)$;

k2: Calculate $x^{l+1} \in \text{argmin} \{G(x)-H_l(x):x \in \mathbb{R}^d\}
~ (P_l)$.

\medskip

To tackle the difficulty due to the large value of $n,$ we first propose the
so called \textit{stochastic DCA} by exploiting the sum structure of $F$.
The basic idea of stochastic DCA is to update, at each iteration, the
minorant of only some randomly chosen $h_{i}$ while keeping the minorant of
the other $h_{i}$. Hence the main advantage of the stochastic DCA versus
standard DCA is the computational reduction in the step of computing a
subgradient of $H$. Meanwhile, the convex subproblem is the same in both
standard DCA and stochastic DCA. The first work in this direction was
published in the conference paper \cite{LeThi17-ICML} where we only considered
a machine learning problem which is a special case of \eqref{sum-model},
namely 
\begin{equation*}
\min_x{\frac{1}{n} \sum_{i=1}^nf_i(x) + \lambda \|x\|_{2,0}},
\end{equation*}
where $f_i$ are L-Lipschitz functions. We rigorously studied the convergence
properties of this stochastic DCA and proved that its convergence is
guaranteed with probability one. In the present work, the same convergence
properties of stochastic DCA for the general model (\ref{sum-model}) is
proved. Furthermore, to deal with the large-scale setting, we propose an 
\textit{inexact stochastic DCA} version in which both subgradient of $H$ and
optimal solution of the resulting convex program are approximately computed.
We show that the convergence properties of stochastic DCA are still valid
for the inexact stochastic DCA.

Finally, we show how to develop the proposed stochastic DCA for the group
variables selection in multi-class logistic regression, a very important
problem in machine learning which takes the form \eqref{sum-model}. 
Numerical experiments on very large synthetic and real-world datasets show
that our approach is more efficient, in both quality and rapidity, than
related methods.

The remainder of the paper is organized as follows. Solution methods based
on stochastic DCA for solving \eqref{sum-model} is developed in Section \ref{sec:SDCA} while the stochastic DCA for the group variables selection in
multi-class logistic regression is presented in Section \ref{sec:application}. Numerical experiments are reported in Section \ref{sec:experiment}.
Finally, Section \ref{Conclusion} concludes the paper.

\section{Stochastic DCA for minimizing a large sum of DC functions}\label{sec:SDCA}

Before presenting the stochastic DCA, let us recall some basic notations
that will be used in the sequel.


The modulus of a convex function
$\theta:\mathbb{R}^d\to\mathbb{R}\cup\{+\infty\}$ on $\Omega$, denoted by $\rho(\theta, \Omega)$ or $\rho(\theta)$ if $\Omega = \mathbb{R}^{n}$, is given by
$$
\rho(\theta, \Omega) = \text{sup} \{ \rho \geq 0 : \theta - (\rho/2) \|. \|^2 \text{ is convex on } \Omega \}.
$$

One says that $\theta$ is $\rho$-convex (resp. \textit{strongly convex}) on $\Omega$ if $\rho(\theta, \Omega) \geq 0$ (resp. $\rho(\theta, \Omega) > 0$).


For $\varepsilon >0$ and $x^0\in \mbox{ dom }\theta ,$ the $\varepsilon $%
-subdifferential of $\theta $ at $x^0$, denoted $\partial \theta _{\varepsilon
}(x^0)$, is defined by 
\begin{equation}
\partial \theta _{\varepsilon }(x^{0}):=\{y\in \mathbb{R}^{d}:\theta (x)\geq
\theta (x^{0})+\langle x-x^{0},y\rangle -\varepsilon :\forall x\in \mathbb{R}^d\},
\label{partial}
\end{equation}%
while $\partial \theta(x^0)$ stands for the usual (or exact) subdifferential of $%
\theta $ at $x^0$ (i.e. $\varepsilon =0$ in (\ref{partial})).

For $\epsilon \geq 0$, a point $x_\epsilon$ is called an $\epsilon$-solution
of the problem $\inf\{f(x): x\in\mathbb{R}^d\}$ if 
\begin{equation*}
f(x_\epsilon) \leq f(x) +\epsilon\ \forall x\in\mathbb{R}^d.
\end{equation*}

\subsection{Stochastic DCA}
Now, let us introduce a stochastic version of DCA, named SDCA, for solving %
\eqref{sum-model}. A natural DC formulation of the problem \eqref{sum-model}
is 
\begin{equation}
\min \left\{ F(x)=G(x)-H(x):x\in \mathbb{R}^{d}\right\} ,  \label{DC_problem}
\end{equation}%
where 
\begin{equation*}
G(x)=\frac{1}{n}\sum_{i=1}^{n}g_{i}(x)\ \text{and}\ H(x)=\frac{1}{n}%
\sum_{i=1}^{n}h_{i}(x).
\end{equation*}%
According to the generic DCA scheme, DCA for solving the problem %
\eqref{DC_problem} consists of computing, at each iteration $l$, a
subgradient $v^{l}\in \partial H(x^{l})$ and solving the convex subproblem
of the form 
\begin{equation}
\min_{\ }\left\{ G(x)-\langle v^{l},x\rangle :x\in \mathbb{R}^{d}\right\} .
\label{convex-dca}
\end{equation}%
As $H=\sum_{i=1}^{n}h_{i}$, the computation of subgradients of $H$ requires
the one of all functions $h_{i}$. This may be expensive when $n$ is very
large. The main idea of SDCA is to update, at each iteration, the minorant
of only some randomly chosen $h_{i}$ while keeping the minorant of the other 
$h_{i}$. Hence, only the computation of such randomly chosen $h_{i}$ is
required.

SDCA for solving the problem \eqref{DC_problem} is described in Algorithm %
\ref{alg:SDCA-DC-sum} below.

\begin{algorithm}[H]
   \caption{SDCA for solving the problem \eqref{sum-model}}
   \label{alg:SDCA-DC-sum}
\begin{algorithmic}
   \State {\bfseries Initialization:} Choose $x^0\in\mathbb{R}^d$, $s_0 = \{1,...,n\}$, and $l\leftarrow 0$.
   \State {\bfseries Repeat}
   \State\quad\quad 1. Compute $v^l_i \in\partial h_i(x^l)$ if $i\in s_l$ and keep $v^l_i=v^{l-1}_i$ if $i\notin s_l$, $l>0$. Set $v^l = \frac{1}{n}\sum_{i=1}^nv^l_i$.
   \State\quad\quad 2. Compute $x^{l+1}$ by solving the convex problem \eqref{convex-dca}.
     
   \State\quad\quad 3. Set $l\leftarrow l+1$ and randomly choose a small subset $s_l\subset\{1,...,n\}$.   
    \State {\bfseries Until} Stopping criterion.
\end{algorithmic}
\end{algorithm}

The following theorem shows that the convergence properties of SDCA are guaranteed with probability one.  
\begin{theorem}\label{theorem1}
Assume that $\alpha^* = \inf F(x) > -\infty$, and $|s_l| = b$ for all $l>0$. Let $\{x^l\}$  be a sequence generated by SDCA , the following statements are hold. 
\begin{itemize}
\item[a)] $\{F(x^l)\}$ is the almost sure convergent sequence.
\item[b)] If $\min_i\rho(h_i)>0$, then $\sum_{l=1}^{\infty}\|x^l-x^{l-1}\|^2 < +\infty$ and $\lim_{l\rightarrow \infty}\|x^l-x^{l-1}\| =0$, almost surely.
\item[c)] If $\min_i\rho(h_i)>0$, then every limit point of $\{x^l\}$ is a critical point of $F$ with probability one.
\end{itemize}
\end{theorem}

\begin{proof}
a) Let $x^0_i$ be the copies of $x^0$. We set $x^{l+1}_i = x^{l+1}$ for all $i\in s_{l+1}$ and $x^{l+1}_j = x^l_j$ for $j\not\in s_{l+1}$. We then have $v^l_i\in\partial h_i(x_i^l)$ for $i=1,...,n$. Let $T_i^l$ be the function given by
\begin{align*}
T_i^l(x) = g_i(x)  - h_i(x_i^l) - \left\langle x - x_i^l, v^l_i \right\rangle.
\end{align*}
It follows from $v^l_i\in\partial h_i(x_i^l)$ that
\begin{equation*}
h_i(x) \geq h_i(x_i^l) + \left\langle x - x_i^l, v^l_i \right\rangle.
\end{equation*}
That implies $T_i^l(x) \geq F_i(x) \geq F_i(x)$ for all $l \geq 0$, $i=1,...,n$. We also observe that $x^{l+1}$ is a solution to the following convex problem
\begin{equation}\label{sub-problem2}
\min_{x} T^l(x):=\frac{1}{n}\sum_{i=1}^nT_i^l(x),
\end{equation}
Therefore
\begin{equation}\label{eq4}
\begin{aligned}
T^l(x^{l+1}) \leq T^l(x^l)& = T^{l-1}(x^l)+ \frac{1}{n}\sum_{i\in s_l}[T_i^l(x^l) - T_i^{l-1}(x^l)] \\
& = T^{l-1}(x^l)+ \frac{1}{n}\sum_{i\in s_l}[F_i(x^l) + 2\epsilon^l - T_i^{l-1}(x^l)],
\end{aligned}
\end{equation}
where the second equality follows from $T_i^l(x^l) = F_i(x^l)$ for all $i\in s_l$.
Let $\mathcal{F}_l$ denote the $\sigma$-algebra generated by the entire history of SDCA up to the iteration $l$, i.e., $\mathcal{F}_0 = \sigma(x^0)$ and $\mathcal{F}_l = \sigma(x^0,...,x^l, s_0,...,s_{l-1})$ for all $l \geq 1$. By taking the expectation of the inequality \eqref{eq4} conditioned on $\mathcal{F}_l$, we have 
\begin{equation*}
\mathbb{E}\left[T^l(x^{l+1})|\mathcal{F}_l\right] \leq T^{l-1}(x^l) - \frac{b}{n}\left[T^{l-1}(x^l) - F(x^l)\right] .
\end{equation*}
By applying the supermartingale convergence theorem \citep{neveu75,Bertsekas03} to the nonnegative sequences $\{T^{l-1}(x^l) - \alpha^*\}, \{\frac{b}{n}[T^{l-1}(x^l) - F(x^l)]\}$ and $\{0\}$, we   conclude that the sequence $\{T^{l-1}(x^l,y^l) - \alpha^*\}$ converges to $T^* -\alpha^*$ and 
\begin{equation}\label{eq728}
\sum_{l=1}^\infty\left[ T^{l-1}(x^l) - F(x^l)\right] < \infty,
\end{equation}
with probability $1$. Therefore $\{F(x^l)\}$ converges almost surely to $T^*$.

b) By $v_i^{l-1}\in\partial h_i(x_i^{l-1})$, we have
\begin{equation*}\label{eq1}
h_i(x)  \geq  h_i(x_i^{l-1}) + \langle x-x_i^{l-1}, v_i^{l-1} \rangle + \frac{\rho(h_i)}{2}\|x - x_i^{l-1}\|^2,\ \forall x\in\mathbb{R}^d.
\end{equation*}
This implies
\begin{equation}\label{eq3}
F_i(x)  \leq T_i^{l-1}(x) - \frac{\rho(h_i)}{2}\|x - x_i^{l-1}\|^2.
\end{equation}
From \eqref{eq4} and \eqref{eq3} with $x=x^l$, we have 
\begin{equation}\label{eq5}
T^l(x^{l+1}) \leq T^{l-1}(x^l) - \frac{1}{n}\sum_{i\in s_l}\frac{\rho(h_i)}{2}\|x - x_i^{l-1}\|^2 .
\end{equation} 
Taking the expectation of the inequality \eqref{eq5} conditioned on $\mathcal{F}_l$, we obtain 
\begin{equation*}
\mathbb{E}\left[T^l(x^{l+1})|\mathcal{F}_l\right] \leq T^{l-1}(x^l) - \frac{b}{4n^2}\sum_{i=1}^n\rho(h_i)\|x^l-x_i^{l-1}\|^2 + \left(\frac{2b}{n}+1
\right)\epsilon^l.
\end{equation*}
Combining this and $\rho = \min_{i=1,...,n}\rho(h_i)  >0$ gives us
\begin{equation*}
\mathbb{E}\left[T^l(x^{l+1})|\mathcal{F}_l\right] \leq T^{l-1}(x^l) - \frac{b\rho}{2n^2}\sum_{i=1}^n\|x^l-x_i^{l-1}\|^2 .
\end{equation*}
Applying the supermartingale convergence theorem to the nonnegative sequences $\{T^{l-1}(x^l) - \alpha^*\}, \{\frac{b\rho}{2 n^2}\sum_{i=1}^n\|x^l-x_i^{l-1}\|^2\}$ and $\{0\}$, we get 
\begin{equation*}\label{eq7}
\sum_{l=1}^\infty\sum_{i=1}^n\|x^l-x_i^{l-1}\|^2 < \infty,
\end{equation*}
with probability $1$. In particular, for $i=1,...,n$, we have
\begin{equation}\label{eq8}
\sum_{l=1}^\infty\|x^l-x_i^{l-1}\|^2 < \infty,
\end{equation}
and hence $\lim_{l \rightarrow \infty}\|x^l-x_i^{l-1}\| =0$ almost surely. 

c) Assume that there exists a sub-sequence $\{x^{l_k}\}$ of $\{x^l\}$ such that $x^{l_k} \rightarrow x^*$ almost surely. From \eqref{eq8}, we have $\|x^{l_k+1}-x_i^{l_k}\| \rightarrow 0$ almost surely. Therefore, by the finite convexity of $h_i$, without loss of generality, we can suppose that the sub-sequence $v_i^{l_k}$ tends to $v^*_i$ almost surely. Since $v_i^{l_k}\in\partial h_i(x_i^{l_k})$ and by the closed property of the subdifferential mapping $\partial h_i$, we have $v^*_i\in\partial h_i(x^*)$. As $x^{l_k+1}$  is a solution of the problem $\min_x T^{l_k}(x)$, we obtain
\begin{equation}\label{eq9}
0 \in \partial T^{l_k}(x^{l_k+1}).
\end{equation}
This is equivalent to
\begin{equation}
0 \in \partial \frac{1}{n}\sum_{i=1}^ng_i(x^{l_k+1}) - \frac{1}{n}\sum_{i=1}^nv_i^{l_k} = \partial G(x^{l_k+1}) -  \frac{1}{n}\sum_{i=1}^nv_i^{l_k}.
\end{equation}
Hence, $\frac{1}{n}\sum_{i=1}^nv_i^{l_k}\in \partial G(x^{l_k+1})$. By the closed property of the subdifferential mapping $\partial G$, we obtain $v^* = \frac{1}{n}\sum_{i=1}^nv^*_i\in \partial G(x^*)$ with probability one. Therefore,
\begin{equation}
v^*\in \partial G(x^*)\cap\partial H(x^*),
\end{equation}
with probability $1$. This implies that $x^*$ is a critical point of $F$ with probability $1$ and the proof is then complete. 
\end{proof}

\subsection{Inexact stochastic DCA}

The SDCA scheme requires the exact computations of $v_i^l$ and $x^{l+1}$.
Observing that, for standard DCA these computations are not necessarily exact~\cite{LeThi18}, we are
suggested to introduce an inexact version of SDCA. This could be useful when the
exact computations of $v_i^l$ and $x^{l+1}$
are expensive. The inexact version of SDCA computes $\epsilon$-subgradients  $v^l_i \in\partial_{\epsilon^l} h_i(x^l)$ and an $\epsilon^l$-solution $x^{l+1}$ of the convex problem \eqref{convex-dca} instead of the exactly computing. The inexact version of SDCA,
named ISDCA, is described as follows.

\begin{algorithm}[H]
   \caption{Inexact SDCA for solving the problem \eqref{sum-model}}
   \label{alg:ISDCA-DC-sum}
\begin{algorithmic}
   \State {\bfseries Initialization:} Choose $x^0\in\mathbb{R}^d$, $s_0 = \{1,...,n\}$, $\epsilon^0\geq 0$ and $l\leftarrow 0$.
   \State {\bfseries Repeat}
   
   \State\quad\quad 1. Compute $v^l_i \in\partial_{\epsilon^l} h_i(x^l)$ if $i\in s_l$ and keep $v^l_i=v^{l-1}_i$ if $i\notin s_l$, $l>0$. Set $v^l = \frac{1}{n}\sum_{i=1}^nv^l_i$.
   \State\quad\quad 2. Compute an $\epsilon^l$-solution $x^{l+1}$ of the convex problem \eqref{convex-dca}.  
  
   \State\quad\quad 3. Set $l\leftarrow l+1$, randomly choose a small subset $s_l\subset\{1,...,n\}$, and update $\epsilon^l\geq 0$.   
    \State {\bfseries Until} Stopping criterion.
\end{algorithmic}
\end{algorithm}

Under an assumption that $\sum_{l=0}^\infty\epsilon^l < +\infty$, the ISDCA has the same convergence properties as SDCA, which are stated in the following theorem. 

\begin{theorem}\label{theorem2}
Assume that $\alpha^* = \inf F(x) > -\infty$, and $|s_l| = b$ for all $l>0$. Let $\{x^l\}$  be a sequence generated by ISDCA with respect to a nonnegative sequence $\{\epsilon^l\}$ such that $\sum_{l=0}^\infty\epsilon^l < +\infty$ almost surely. The following statements are hold. 
\begin{itemize}
\item[a)] $\{F(x^l)\}$ is the almost sure convergent sequence.
\item[b)] If $\min_i\rho(h_i)>0$, then $\sum_{l=1}^{\infty}\|x^l-x^{l-1}\|^2 < +\infty$ and $\lim_{l\rightarrow \infty}\|x^l-x^{l-1}\| =0$, almost surely.
\item[c)] If $\min_i\rho(h_i)>0$, then every limit point of $\{x^l\}$ is a critical point of $F$ with probability one.
\end{itemize}
\end{theorem}
This theorem is analogously proved as Theorem \ref{theorem1}  and its proof is provided in Appendix \ref{app}.

\section{Application to Group Variables Selection in multi-class Logistic Regression}\label{sec:application}

Logistic regression, introduced by D. Cox in 1958 \cite{Cox1958}, is undoubtedly one of the most popular supervised learning methods. Logistic regression has been successfully applied in various real-life problems such as cancer detection \cite{Kim2008}, medical \cite{Boyd1987,Bagley2001,Subasi2005}, social science \cite{King2001}, etc. Especially, logistic regression combined with feature selection has been proved to be suitable for high dimensional problems, for instance, document classification \cite{Genkin2007} and microarray classification \cite{Liao2007,Kim2008}.

The multi-class logistic regression problem can be described as follows. Let $\{(x_i,y_i): i = 1,...,n\}$ be a training set with observation vectors $x_i\in\mathbb{R}^d$ and labels $y_i\in \{1,...,Q\}$ where $Q$ is the number of classes. Let $W$ be the $d\times Q$ matrix whose columns are $W_{:,1},...,W_{:,Q}$ and $b=(b_1,...,b_Q)\in\mathbb{R}^Q$. The couple $(W_{:,i},b_i)$ forms the hyperplane $f_i:=W_{:,i}^Tx + b_i$+ that separates the class $i$ from the other classes.

In the multi-class logistic regression problem, the conditional probability $p(Y=y|X=x)$ that an instance $x$ belongs to a class $y$ is defined as
\begin{equation}
p(Y=y|X=x) = \frac{\exp(b_y+W_{:,y}^Tx)}{\sum\limits_{k=1}^Q\exp(b_k+W_{:,k}^Tx)}.
\end{equation}
We aim to find $(W,b)$ for which the total probability of the training observations $x_i$ belonging to its correct classes $y_i$ is maximized. A natural way to estimate $(W,b)$ is to minimize the negative log-likelihood function which is defined by
\begin{equation}
\mathcal{L}(W,b) := \frac{1}{n}\sum\limits_{i=1}^n\ell(x_i,y_i,W,b)
\end{equation}
where $\ell(x_i,y_i,W,b)  = -\log p(Y=y_i|X=x_i)$. Moreover, in high-dimensional settings, there are many irrelevant and/or redundant features. 
Hence, we need to select important features in order to reduce overfitting of the training data. 
A feature $j$ is to be removed if and only if all components in the row $j$ of $W$ are zero. Therefore, it is reasonable to consider rows of $W$ as groups. Denote by $W_{j,:}$ the $j$-th row of the matrix $W$. The $\ell_{q,0}$-norm of $W$, i.e., the number of non-zero rows of $W$, is defined by 
\begin{equation*}
\|W\|_{q,0} = |\{j\in\{1,...,d\}:\|W_{j,:}\|_q \neq 0\}|.
\end{equation*}
Hence, the $\ell_{q,0}$ regularized multi-class logistic regression problem is formulated as follows
\begin{equation}\label{slr}
\min_{W,b}\left\{ \frac{1}{n}\sum_{i=1}^n\ell(x_i,y_i,W,b) + \lambda\|W\|_{q,0}\right\}.
\end{equation}

In this application, we use a non-convex approximation of the $\ell_{q,0}$-norm based on the following two penalty functions $\eta_\alpha(s)$:
\begin{eqnarray*}
\text{Exponential:} & 
\eta^{\text{exp}}_\alpha(s) &= 1-\exp(-\alpha s),
\\
\text{Capped-$\ell_1$:} &
\eta^{\text{cap-$\ell_1$}}_\alpha(s) &= \min\{1, \alpha s\}.
\end{eqnarray*}
These penalty functions have shown their efficiency in several problems, for instance, individual variables selection in SVM \cite{Bradley98,LT08}, sparse optimal scoring problem \cite{LT16}, sparse covariance matrix estimation problem \cite{phan2017}, and bi-level/group variables selection \cite{LeThi2019,Phan2019}. The corresponding approximate problem of \eqref{slr} takes the form:
\begin{equation}\label{aslr}
\min_{W,b}\left\{ \frac{1}{n}\sum_{i=1}^n\ell(x_i,y_i,W,b) + \lambda\sum_{j=1}^d\eta_\alpha(\|W_{j,:}\|_q)\right\}.
\end{equation}
Since $\eta_\alpha$ is increasing on $[0,+\infty)$, the problem \eqref{aslr} can be equivalently reformulated as follows
\begin{equation}\label{aslrequivalent}
\min_{(W,b,t)}\left\{ \frac{1}{n}\sum_{i=1}^n\left[\ell(x_i,y_i,W,b) + \chi_\Omega(W,b,t) +  \lambda\sum_{j=1}^d\eta_\alpha(t_j)\right]\right\},
\end{equation}
where $\Omega = \{(W,b,t)\in \mathbb{R}^{d\times Q}\times\mathbb{R}^Q\times\mathbb{R}^d: \|W_{j,:}\|_q \leq t_j, j = 1,...,d\}$. 
Moreover, as $\ell(x_i,y_i,W,b)$ is differentiable with $L$-Lipschitz continuous gradient and $\eta_\alpha$ is concave, the problem \eqref{aslrequivalent} takes the form of \eqref{sum-model} where the function $F_i(W,b,t)$ is given by
\begin{equation*}
F_i(W,b,t) = \ell(x_i,y_i,W,b) + \chi_\Omega(W,b,t) + \lambda\sum_{j=1}^d\eta_\alpha(t_j):= g_i(W,b,t) - h_i(W,b,t),
\end{equation*}
where the DC components $g_i$a and $h_i$ are defined by
\begin{align*}
g_i(W,b,t) &= \frac{\rho}{2}\|(W,b)\|^2 +\chi_\Omega(W,b,t),\\
h_i(W,b,t) &= \frac{\rho}{2}\|(W,b)\|^2 - \ell(x_i,y_i,W,b) - \lambda\sum_{j=1}^d\eta_\alpha(t_j),
\end{align*}
with $\rho >L$.

Before presenting SDCA for solving the problem~\eqref{aslrequivalent}, let us show how to apply standard DCA on this problem.

\subsection{Standard DCA for solving the problem \eqref{aslrequivalent}}

We consider three norms corresponding to $q \in \{1,2,\infty\}$. 
DCA applied to \eqref{aslrequivalent} consists of computing,
at each iteration $l$, 
$(U^l,v^l,z^l)\in\partial H(W^l,b^l,t^l)$, and solving the convex sub-problem
\begin{equation}\label{convex_DCA}
\min_{(W,b,t)}\left\{\frac{\rho}{2}\|(W,b)\|^2 +\chi_\Omega(W,b,t) - \langle U^l, W\rangle - \langle v^l, b\rangle - \langle z^l, t\rangle \right\}.
\end{equation}
The computation of $(U^l,v^l,z^l)$ is explicitly defined as follows.
$$(U^l,v^l,z^l) = \frac{1}{n}\sum_{i=1}^n(U_i^l,v_i^l,z_i^l), (U_i^l,v_i^l,z_i^l)\in\partial h_i(W^l,b^l,t^l).$$
More precisely
\begin{equation}\label{gradiensdca}
\begin{aligned}
(U_i^l)_{:,k} &= \rho W_{:,k}^l - \left(p_k^l(x_i) - \delta_{ky_i}\right)x_i, k = 1,...Q,\\
(v_i^l)_k &= \rho b_k^l - \left(p_k^l(x_i) - \delta_{ky_i}\right), k = 1,...Q,\\
(z_i^l)_j & = \left\{\begin{matrix}
-\lambda\alpha\exp(-\alpha t_j^l), &j = 1,\dots,d 
&\text{ if } \eta_\alpha= \eta^{\text{exp}}_\alpha,
\\ 
-\lambda\alpha \text{ if } \alpha t_j^l \leq 1, \text{ and }0 \text { otherwise}, 
&j = 1,\dots,d,
&\text{ if } \eta_\alpha = \eta^{\text{cap}-\ell_1}_\alpha,
\end{matrix}\right.
\end{aligned}
\end{equation}
with $p^l_k(x_i)=\exp(b_k^l+(W^l_{:,k})^Tx_i)/(\sum_{h=1}^Qb_h^l+(W^l_{:,h})^Tx_i))$, $\delta_{ky_i}= 1$ if $k=y_i$ and $0$ otherwise. 

The convex sub-problem~\eqref{convex_DCA} can be solved as follows (note that $z_j^l \leq 0$ for $j = 1, \dots, d)$
\begin{align}
W^{l+1} &= \arg\min_{W}\left\{\frac{\rho}{2}\|W\|^2 +\sum_{j=1}^d(-z_j^l)\|W_{j,:}\|_q - \langle U^l, W\rangle  \right\},\label{convex_DCA2}\\
b^{l+1} &=\arg\min_{b}\left\{\frac{\rho}{2}\|b\|^2- \langle v^l, b\rangle \right\} = \frac{1}{\rho}v^l, \label{comute_b}\\
t_j^{l+1}& = \|W_{j,:}^{l+1}\|_q, j = 1,...,d. \label{comute_t}
\end{align}
Since the problem~\eqref{convex_DCA2} is separable in rows of $W$, solving it amounts to solving $d$ independent sub-problems
\begin{equation*}\label{sub_convex1}
W_{j,:}^{l+1}=\arg\min_{W_{j,:}}\left\{\frac{\rho}{2}\|W_{j,:}\|^2 + (-z_j^l)\|W_{j,:}\|_q - \langle U^l_{j,:}, W_{j,:}\rangle  \right\}.
\end{equation*}
Moreover, $W_{j,:}^{l+1}$ is computed via the following proximal operator
\begin{equation*}\label{DCA_eq1}
W_{j,:}^{l+1} = \prox_{(-z_j^l)/\rho\|\cdot\|_q}\left(U^l_{j,:}/\rho\right),
\end{equation*}
where the proximal operator $\prox_f(\nu)$ is defined by
\begin{equation*}
\prox_f(\nu) = \argmin_t\left\{\frac{1}{2}\|t-\nu\|^2 + f(t)\right\}.
\end{equation*}
The proximal operator of $(-z_j^l)/\rho\|\cdot\|_q$ can be efficiently computed \citep{Parikh14}. The computation of $\prox_{(-z_j^l)/\rho\|.\|_q}\left(\nu/\rho\right)$ can be summarized in Table \ref{prox}.
\begin{table}[bt]
\caption{Computation of $W^{l+1}_{j,:}= \prox_{(-z_j^l)/\rho\|.\|_q}\left(U^l_{j,:}/\rho\right)$ corresponding to $q\in\{1,2,\infty\}$.}
\label{prox} 
\vskip 0.15in 
\begin{center}
\begin{tabular}{ll}
\hline
\noalign{\smallskip} 
$q$ &  $\prox_{(-z_j^l)/\rho\|.\|_q}\left(U^l_{j,:}/\rho\right)$\\ 

\noalign{\smallskip}\hline
 
\noalign{\smallskip}
$1$ &  $\left(|U^l_{j,:}|/\rho - (-z_j^l)/\rho\right)_+\circ\text{sign}(U^l_{j,:})$\\
\noalign{\smallskip}\noalign{\smallskip}
$2$ & $\begin{cases}
\left(1-\frac{-z_j^l}{\|U^l_{j,:}\|_2}\right)U^l_{j,:}/\rho\ &\text{if}\ \|U^l_{j,:}\|_2 > -z_j^l\\
0\ &\text{if}\ \|U^l_{j,:}\|_2 \leq -z_j^l.
\end{cases}$\\

\noalign{\smallskip}\noalign{\smallskip} 
$\infty$ &  $\begin{cases}
U^l_{j,:}/\rho - \left(\frac{1}{-z_j^l}|U^l_{j,:}|-\delta\right)_+\circ \text{sign}(U^l_{j,:}) \ &\text{if}\ \|U^l_{j,:}\|_1 > -z_j^l\\
0\ &\text{if}\ \|U^l_{j,:}\|_1 \leq -z_j^l,
\end{cases}$\\

 & where $\delta$ satisfies $\sum_{k=1}^Q\left(\frac{1}{-z_j^l}|U^l_{j,k}|-\delta\right)_+ = 1.$ \\

\noalign{\smallskip}\hline
\end{tabular}
\end{center}
\end{table}
DCA based algorithms for solving \eqref{aslrequivalent} with $q\in\{1,2,\infty\}$ are described as follows.

\noindent\makebox[\linewidth]{\rule{\textwidth}{0.1pt}}
\label{DCA-Full}
\textbf{DCA-$\ell_{q,0}$:} DCA for solving \eqref{aslrequivalent} with $q\in\{1,2,\infty\}$\\
\vskip -0.25in
\noindent\makebox[\linewidth]{\rule{\textwidth}{0.1pt}}
\vskip -0.05in
\begin{algorithmic}
   \State {\bfseries Initialization:} Choose $(W^0,b^0)\in\mathbb{R}^{d\times Q}\times\mathbb{R}^Q$, $\rho > L$ and $l\leftarrow 0$.
   \State {\bfseries Repeat}
   \State\quad\quad 1. Compute $(U^l,v^l,z^l) = \frac{1}{n}\sum_{i=1}^n(U_i^l,v_i^l,z_i^l)$, where $(U_i^l,v_i^l,z_i^l)$, $i=1,...,n$ are defined in
    \eqref{gradiensdca}.
   
   \State\quad\quad 2. Compute $(W^{l+1},b^{l+1},t^{l+1})$ 
    according to Table~\ref{prox}, \eqref{comute_b} and \eqref{comute_t}, respectively.
   
	\State\quad\quad 3. $l\leftarrow l+1$.
    \State{\bfseries Until} Stopping criterion.
\end{algorithmic}
\vskip -0.1in
\noindent\makebox[\linewidth]{\rule{\textwidth}{0.1pt}}\\

\subsection{SDCA for solving the problem \eqref{aslrequivalent}}
In SDCA, at each iteration $l$, we have to compute $(U_i^l,v_i^l,z_i^l)\in\partial h_i(W^l,b^l,t^l)$ for $i\in s_l$ and keep $(U_i^l,v_i^l,z_i^l)=(U_i^{l-1},v_i^{l-1},z_i^{l-1})$ for $i\notin s_l$, where $s_l$ is a randomly chosen subset of the indexes, and solve the convex sub-problem taking the form of \eqref{convex_DCA}. 
Hence, SDCA for solving~\eqref{aslrequivalent} is described below.

\noindent\makebox[\linewidth]{\rule{\textwidth}{0.1pt}}
\textbf{SDCA-$\ell_{q,0}$:} SDCA for solving \eqref{aslrequivalent} with $q\in\{1,2,\infty\}$\\
\vskip -0.25in
\noindent\makebox[\linewidth]{\rule{\textwidth}{0.1pt}}
\vskip -0.05in
\begin{algorithmic}
   \State {\bfseries Initialization:} Choose $(W^0,b^0)\in\mathbb{R}^{d\times Q}\times\mathbb{R}^Q$, $t^0_j = \|W_{j,:}^0\|_q$, $\rho > L$, $s_0=\{1,...,n\}$ and $l\leftarrow 0$.
   \State {\bfseries Repeat}
    \State\quad\quad 1. Compute $(U_i^l,v_i^l,z_i^l)$ by \eqref{gradiensdca} if $i\in s_l$ and keep $(U_i^l,v_i^l,z_i^l)=(U_i^{l-1},v_i^{l-1},z_i^{l-1})$ if $i\notin s_l$. Set $(U^l,v^l,z^l) = \frac{1}{n}\sum_{i=1}^n(U_i^l,v_i^l,z_i^l)$.
      \State\quad\quad 2. Compute $(W^{l+1},b^{l+1},t^{l+1})$ 
    according to Table~\ref{prox}, \eqref{comute_b} and \eqref{comute_t}, respectively.
   
	\State\quad\quad 3. $l\leftarrow l+1$ and randomly choose a small subset $s_l\subset\{1,...,n\}$.
    \State{\bfseries Until} Stopping criterion.
\end{algorithmic}
\vskip -0.1in
\noindent\makebox[\linewidth]{\rule{\textwidth}{0.1pt}}\\
\section{Numerical Experiment}\label{sec:experiment}


\subsection{Datasets} \label{sec:dataset}
To evaluate the performances of algorithms, we performed numerical experiments on two types of data: real datasets (\textit{covertype}, \textit{madelon}, \textit{miniboone}, \textit{protein}, \textit{sensit} and \textit{sensorless}) and simulated datasets (\textit{sim\_1}, \textit{sim\_2} and \textit{sim\_3}). All real-world datasets are taken from the well-known UCI and LibSVM data repositories. We give below a brief description of real datasets:
\begin{itemize}
\item \textit{covertype} belongs to the Forest Cover Type Prediction from strictly cartographic variables challenge\footnote{\url{https://archive.ics.uci.edu/ml/datasets/Covertype}}. It is a very large dataset containing $581,012$ points described by $54$ variables.
\item \textit{madelon} is one of five datasets used in the NIPS 2003 feature selection challenge\footnote{\url{https://archive.ics.uci.edu/ml/datasets/Madelon}}. The dataset contains $2600$ points, each point is represented by $500$ variables. Among $500$ variables, there are only $5$  informative variables and $15$ redundant variables (which are created by linear combinations of $5$ informative variables). The $480$ others variables were added and have no predictive power. Notice that \textit{madelon} is a highly non-linear dataset.
\item \textit{miniboone} is taken form the MiniBooNE experiment to observe neutrino oscillations\footnote{\url{https://archive.ics.uci.edu/ml/datasets/MiniBooNE+particle+identification}}, containing $130,065$ data points.	
\item \textit{protein}  \footnote{\url{https://www.csie.ntu.edu.tw/~cjlin/libsvmtools/datasets/multiclass.html} \label{dataset_liblinear}} is a dataset for classifying protein second structure state ($\alpha$, $\beta$, and coil) of each residue in amino acid sequences, including $24,387$ data points.
\item \textit{sensit} \textsuperscript{\ref{dataset_liblinear}} dataset obtained from distributed sensor network for vehicle classification. It consists of $98,528$ data points categorized into 3 classes: Assault Amphibian Vehicle (AAV), Dragon Wagon (DW) and noise. 

\item \textit{sensorless} measures electric current drive signals from different operating conditions, which is classified into 11 different classes \footnote{\url{https://archive.ics.uci.edu/ml/datasets/Dataset+for+Sensorless+Drive+Diagnosis}}. It is a huge dataset, which contains $58,509$ data points, described by $48$ variables.
\end{itemize}

We generate three synthetic datasets (\textit{sim\_1}, \textit{sim\_2} and \textit{sim\_3}) by the same process proposed in \cite{Witten2011}. 
In the first dataset (\textit{sim\_1}), variables are independent and have different means in each class. In dataset (\textit{sim\_2}), variables also have different means in each class, but they are dependent. 
The last synthetic dataset (\textit{sim\_3}) has different one-dimensional means in each class with independent variables. 
Detail produces to generate three simulated datasets are described as follows:

\begin{itemize}
\item For \textit{sim\_1}: we generate a four-classes classification problem. Each class is assumed to have a multivariate normal distribution $\mathcal{N}(\mu_k,I)$, $k=1,2,3,4$ with dimension of $d=50$. The first $10$ components of $\mu_1$ are $0.5$, $\mu_{2j}=0.5$ if $11\leq j\leq 20$, $\mu_{3j}=0.5$ if $21\leq j\leq 30$, $\mu_{4j}=0.5$ if $31\leq j\leq 40$ and $0$ otherwise. We generate $250,000$ instances with equal probabilities.
\item For \textit{sim\_2}: this synthetic dataset contains three classes of multivariate normal distributions $\mathcal{N}(\mu_{k},\Sigma)$, $k = 1, 2, 3$, each of dimension $d=50$. 
The components of $\mu_1 = 0$, $\mu_{2j} = 0.4$ and $\mu_{3j}=0.8$ if $j\leq 40$ and $0$ otherwise. 
The covariance matrix $\Sigma$ is the block diagonal matrix with five blocks of dimension $10\times 10$ whose element $(j,j')$ is $0.6^{|j-j'|}$. We generate $150,000$ instances.
\item For \textit{sim\_3}: this synthetic dataset consists of four classes. For class $k=1,2,3,4$, 
$i\in C_k$ then $X_{ij} \sim \mathcal{N}(0,1)$ for $j > 100$, 
and 
$X_{ij} \sim \mathcal{N}(\frac{k-1}{3},1)$ otherwise, 
where 
$\mathcal{N}(\mu,\sigma^2)$ denotes the Gaussian distribution with mean $\mu$ and variance $\sigma^2$. 
We generate $62,500$ data points for each class.
\end{itemize}

The number of points, variables and classes of each dataset are summarized in the first column of Table~\ref{tbl.experiment}.


\subsection{Comparative algorithms}\label{sec:comparative_algo}

To the best of our knowledge, there is no existing method in the literature for solving the group variable selection in multi-class logistic regression using $\ell_{q,0}$ regularization. However, closely connected to the Lasso ($\ell_1$-norm), \cite{Vincent2014} proposed to use the convex regularization $\ell_{2,1}$ instead of $\ell_{2,0}$. Thus, the resulting problem takes the form
\begin{equation}
\label{logisitic_l21}
\min\limits_{W,b}\left\{ \frac{1}{n}\sum\limits_{i=1}^n\ell(x_i,y_i,W,b) + \lambda\|W\|_{2,1}\right\} .
\end{equation}
A coordinate gradient descent, named \texttt{msgl}, was proposed in \cite{Vincent2014} to solve the problem \eqref{logisitic_l21}. 
\texttt{msgl} is a comparative algorithm in our experiment.

On another hand, we are interested in a comparison between our algorithms and a stochastic based method. A stochastic gradient descent algorithm to solve~\eqref{logisitic_l21}, named $\texttt{SPGD-\textit{$\ell_{2, 1}$}}$,  is developed for this purpose. $\texttt{SPGD-\textit{$\ell_{2, 1}$}}$ is described as follows.


\noindent\makebox[\linewidth]{\rule{\textwidth}{0.1pt}}
\textbf{SPGD-$\ell_{2,1}$:} Stochastic Proximal Gradient Descent for solving \eqref{logisitic_l21}\\
\vskip -0.25in
\noindent\makebox[\linewidth]{\rule{\textwidth}{0.1pt}}
\vskip -0.05in
\begin{algorithmic} \label{SPGD-algo}
    \State {\bfseries Initialization:} Choose $(W^0,b^0)\in\mathbb{R}^{d\times Q}\times\mathbb{R}^Q$, and $l\leftarrow 0$.
   \State {\bfseries Repeat}
   \State\quad\quad 1. Randomly choose a small subset $s_l \subset\{1,...,n\}$.  
   Set $\alpha_{l} = \frac{n}{10 l}$. 
   Compute $\bar{U}_{:,k}^l = W_{:,k}^l - \frac{\alpha_{l}}{|s_{l}|}\sum_{i \in s_{l}} \left(p_k^l(x_i) - \delta_{ky_i}\right)x_i, k = 1,...Q$.
   
   \State\quad\quad 2. Compute $(W^{l+1},b^{l+1})$ by
   \begin{equation}\label{spgd_solution}
	\begin{aligned}
W^{l+1}_{j,:} &= \left( \| \bar{U}^l_{j,:} \|_2 - \alpha_{l} \lambda \right)_+ \frac{\bar{U}^l_{j,:} }{\| \bar{U}^l_{j,:}  \|_2}, j = 1, ..., d\\
b^{l+1}_{k} &= b^{l}_{k}  - \frac{\alpha_{l}}{|s_{l}|} \sum_{i \in s_l} \left(p_k^l(x_i) - \delta_{ky_i}\right), k = 1,..., Q.
	\end{aligned}
   \end{equation}

	\State\quad\quad 3. $l\leftarrow l+1$.
    \State{\bfseries Until} Stopping criterion.
    
\end{algorithmic}
\vskip -0.1in
\noindent\makebox[\linewidth]{\rule{\textwidth}{0.1pt}}\\

\subsection{Experiment setting}

We randomly split each dataset into a training set and a test set. The training set contains 80\% of the total number of points and the remaining 20\% are used as test set.

In order to evaluate the performance of algorithms, we consider the following three criteria: the classification accuracy (percentage of well classified point on test set), the sparsity of obtained solution and the running time (measured in seconds). The sparsity is computed as the percentage of selected variables. Note that a variable $j \in \left\lbrace 1, \ldots, d \right\rbrace$ is considered to be removed if all components of the row $j$ of $W$ are smaller than a threshold, i.e., $\left|W_{j,i}\right| \le 10^{-8}, \forall i \in {1, \ldots, Q}$. We perform each algorithm 10 times and report the mean and standard deviation of each criterion.

We use the early-stopping condition for \texttt{SDCA} and \texttt{SPGD-$\ell_{2, 1}$}. Early-stopping is a well-know technique in machine learning, especially in stochastic learning which permits to avoid the over-fitting in learning. More precisely, after each epoch, we compute the classification accuracy on a validation set which contains 20\% randomly chosen data points of training set. We stop \texttt{SDCA} and \texttt{SPGD-$\ell_{2, 1}$} if the classification accuracy is not improved after $n_{patience} = 5$ epochs. The batch size of stochastic algorithms (\texttt{SDCA} and \texttt{SPGD-$\ell_{2, 1}$}) is set to $10\%$. \texttt{DCA} is stopped if the difference between two consecutive objective functions is smaller than a threshold $\epsilon_{stop} = 10^{-6}$. For \textit{msgl}, we use its default stopping parameters as in \citep{Vincent2014}. We also stop algorithms if they exceed $2$ hours of running time in the training process.

The parameter $\alpha$ for controlling the tightness of zero-norm approximation is chosen in the set  $\left\lbrace 0.5, 1, 2, 5\right\rbrace$. 
We use the solution-path procedure for the trade-off parameter $\lambda$. Let $\lambda_{1} > \lambda_{2} > ... > \lambda_{l}$ be a decreasing sequence of $\lambda$. 
At step $k$, we solve the problem \eqref{slr} with $\lambda = \lambda_k$ from the initial point chosen as the solution of the previous step $k-1$. Starting with a large value of $\lambda$, we privilege the sparsity of solution (i.e. selecting very few variables) over the classification ability. Then by decreasing the value $\lambda$ decreases, we select more variables in order to increase the classification accuracy. 
In our experiments, the sequence of $\lambda$ is set to $\{10^4, 3\times 10^3, 10^3,\ldots,3\times 10^{-3}, 10^{-3}\}$.

All experiments are performed on a PC Intel (R) Xeon (R) E5-2630 v2 @2.60 GHz with 32GB RAM. 

\subsection{Experiment 1} \label{exp_1}

In this experiments, we study the effectiveness of \texttt{SDCA}. For this purpose, we choose the $\ell_{2,0}$ regularization, and perform a comparison between \texttt{SDCA-$\ell_{2, 0}$-exp} and \texttt{DCA-$\ell_{2, 0}$-exp}. Furthermore, we will compare \texttt{SDCA-$\ell_{2, 0}$-exp} with \texttt{msgl} and \texttt{SPGD-$\ell_{2, 1}$}, two algorithms for solving the multi-class logistic regression using $\ell_{2, 1}$ regularization (c.f Section~\ref{sec:comparative_algo}).

The comparative results between are reported in Table~\ref{tbl.experiment} and Figure~\ref{fig:exp_1}. Note that the running time is plotted in logarithmic scale.

\begin{figure}[tbh!]
	\centering
\subfigure{	
	\includegraphics[width=\linewidth]{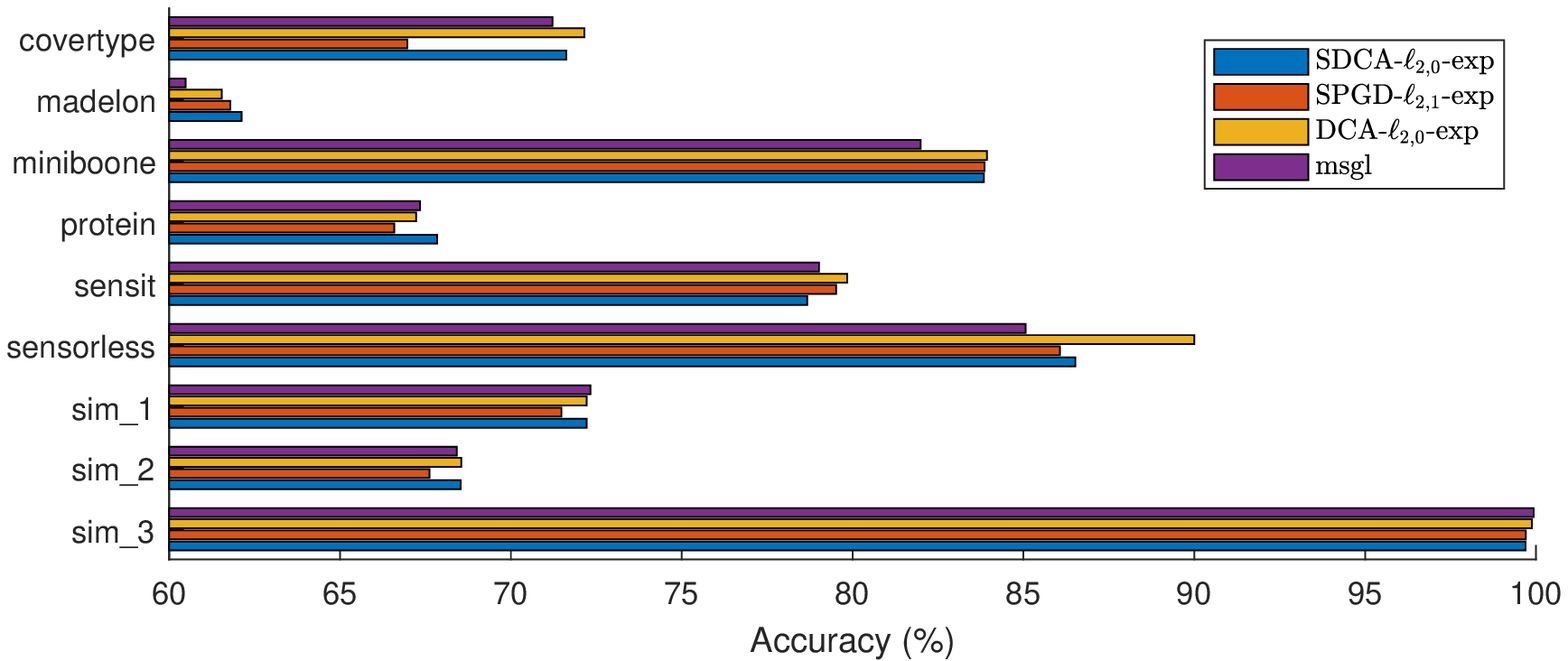}	
}
\subfigure{	
	\includegraphics[width=\linewidth]{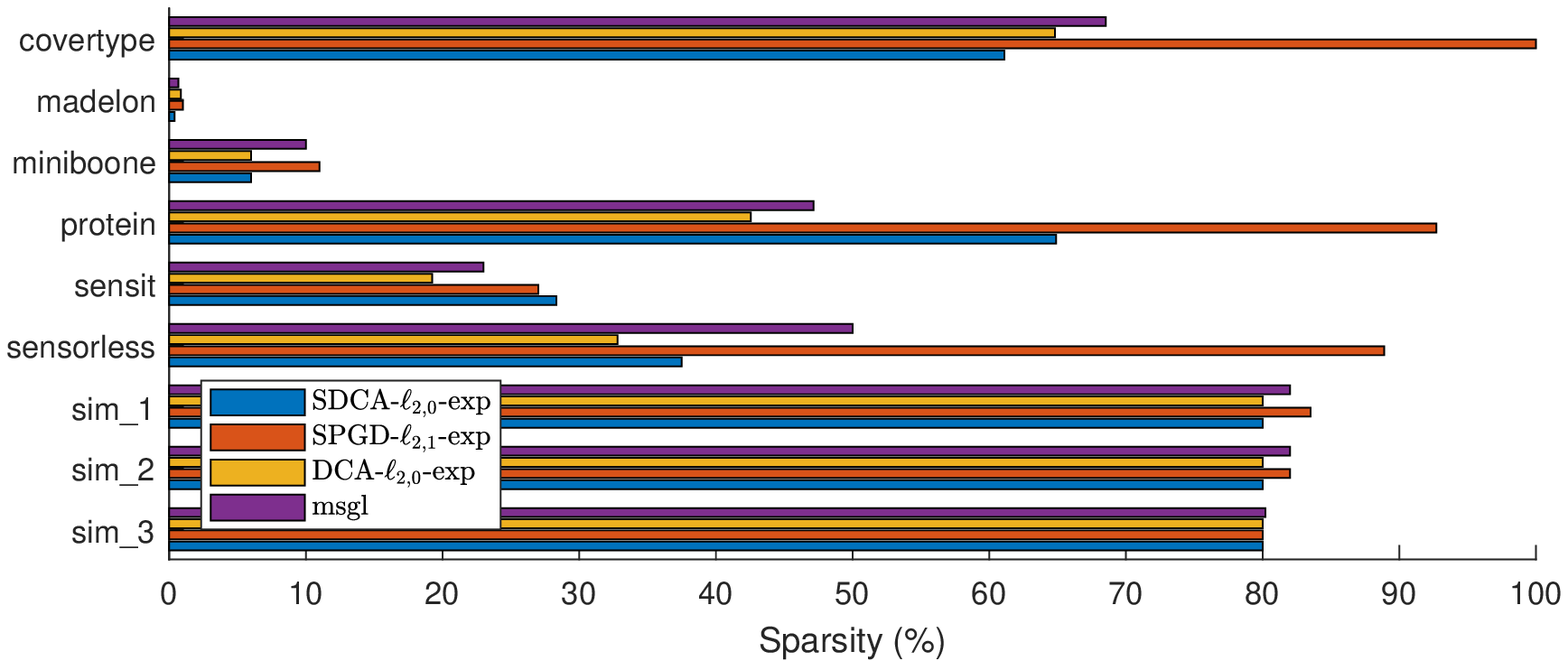}	
}
\subfigure{	
	\includegraphics[width=\linewidth]{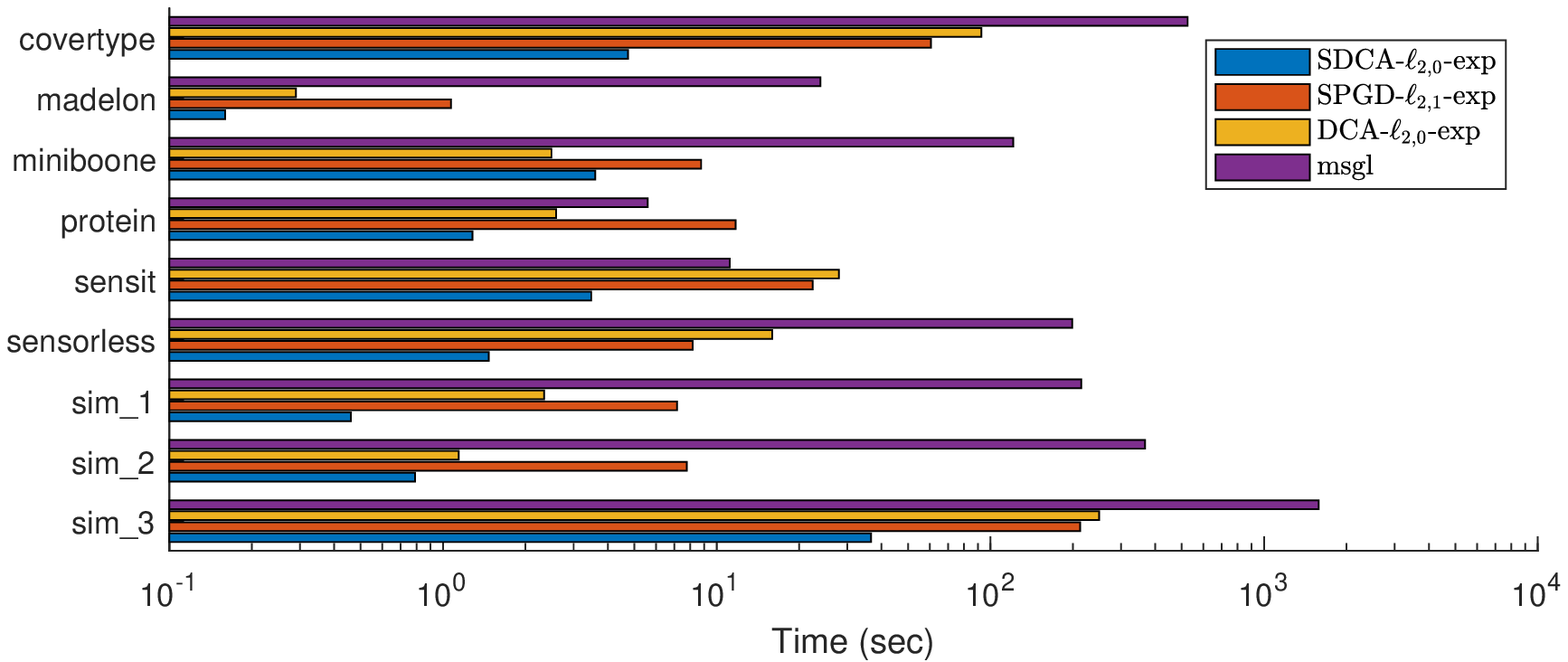} 	
}
\caption{Comparative results between \texttt{SDCA-$\ell_{2, 0}$-exp}, \texttt{DCA-$\ell_{2, 0}$-exp}, \texttt{SPGD-$\ell_{2, 1}$} and \texttt{msgl} (running time is plotted on a logarithmic scale).}
\label{fig:exp_1}
\end{figure}

\noindent \textbf{Comparison between \texttt{SDCA-$\ell_{2, 0}$-exp} and \texttt{DCA-$\ell_{2, 0}$-exp}}

\sloppy In term of classification accuracy, \texttt{SDCA-$\ell_{2, 0}$-exp} produces fairly similar result comparing with \texttt{DCA-$\ell_{2, 0}$-exp}. 
\texttt{DCA-$\ell_{2, 0}$-exp} is better than \texttt{SDCA-$\ell_{2, 0}$-exp} on $4$ datasets (\textit{covertype}, \textit{sensit}, \textit{sensorless} and \textit{sim\_3}) while \texttt{SDCA-$\ell_{1, 0}$-exp} gives better results on $2$ datasets (\textit{madelon} and \textit{protein}). The two biggest gaps ($3.49\%$ and $1.17\%$) occur on dataset \textit{sensorless} and \textit{sensit} respectively.

\sloppy As for the sparsity of solution, \texttt{DCA-$\ell_{2, 0}$-exp} and \texttt{SDCA-$\ell_{2, 0}$-exp} provide the same results on $4$ datasets (\textit{miniboon}, \textit{sim\_1}, \textit{sim\_2} and \textit{sim\_3}). \texttt{DCA-$\ell_{2, 0}$-exp} suppresses more variables than \texttt{SDCA-$\ell_{2, 0}$-exp} on $3$ datasets (\textit{protein}, \textit{sensit} and \textit{sensorless}), while \texttt{SDCA-$\ell_{2, 0}$-exp} gives better sparsity on \textit{covertype} and \textit{madelon}. The gain of \texttt{DCA-$\ell_{2, 0}$-exp} on this criterion is quite high, up to $22.3\%$ on dataset \textit{protein}.

\sloppy Concerning the running time, \texttt{SDCA-$\ell_{2, 0}$-exp} clearly outperforms \texttt{DCA-$\ell_{2, 0}$-exp}. 
Except for \textit{miniboone} where \texttt{DCA-$\ell_{2, 0}$-exp} is $1.11$ second faster, the gain of \texttt{SDCA-$\ell_{2, 0}$-exp} is huge.  \texttt{SDCA-$\ell_{2, 0}$-exp} is up to $19.58$ times faster than \texttt{DCA-$\ell_{2, 0}$-exp} (dataset \textit{covertype}).

\sloppy Overall, \texttt{SDCA-$\ell_{2, 0}$-exp} is able to achieve equivalent classification accuracy with a running time much smaller than \texttt{DCA-$\ell_{2, 0}$-exp}. 

\noindent \textbf{Comparison between \texttt{SDCA-$\ell_{2, 0}$-exp} and \texttt{msgl}}.

\texttt{SDCA-$\ell_{2, 0}$-exp} provides better classification accuracy on $6$ out of $9$ datasets with a gain up to $1.85\%$. For the $3$ remaining datasets, the gain of \texttt{msgl} in accuracy is smaller than $0.3\%$. As for the sparsity of solution, the two algorithms are comparable. \texttt{SDCA-$\ell_{2, 0}$-exp} is by far faster than \texttt{msgl} on all datasets, from $3.2$  times to $470$ time faster. 

\noindent \textbf{Comparison between \texttt{SDCA-$\ell_{2, 0}$-exp} and \texttt{SPGD-$\ell_{2, 1}$}}.

In term of classification accuracy, SDCA is better on $6$ datasets with a gain up to $4.65\%$, whereas SPGD only gives better result on \textit{sensit}. Moreover, the number of selected variables by \texttt{SPGD-$\ell_{2, 1}$} is considerably higher. 
\texttt{SPGD-$\ell_{2, 1}$} chooses from $2\%$ to $51.39\%$  more variables than \texttt{SDCA} in $6$ over $9$ cases (\textit{covertype}, \textit{miniboone}, \textit{protein},  \textit{sensorless}, \textit{sim\_1}, and \textit{sim\_2}), and $> 27\%$ more in $3$ over $9$ cases (\textit{covertype}, \textit{protein} and \textit{sensorless}). As for the running time, \texttt{SDCA-$\ell_{2, 0}$-exp} is up to $15.68$ times faster than \texttt{SPGD-$\ell_{2, 1}$}. Overall, \texttt{SDCA-$\ell_{2, 0}$-exp} clearly outperforms \texttt{SPGD-$\ell_{2, 1}$} on all three criteria. 

In conclusion, as expected, \texttt{SDCA-$\ell_{2, 0}$-exp} reduces considerably the running time of \texttt{DCA-$\ell_{2, 0}$-exp} while achieving equivalent classification accuracy. Moreover, \texttt{SDCA-$\ell_{2, 0}$-exp} outperforms the two related algorithms \texttt{msgl} and \texttt{SPGD-$\ell_{2, 1}$}.

\subsection{Experiment 2} \label{exp_2}

\sloppy In this experiment, in order to study the effectiveness of different non-convex regularizations $\ell_{q,0}$, we compare three algorithms \texttt{SDCA-$\ell_{1, 0}$-exp}, \texttt{SDCA-$\ell_{2, 0}$-exp} and \texttt{SDCA-$\ell_{\infty, 0}$-exp}. The results are reported in Table \ref{tbl.experiment} and plotted in Figure~\ref{fig:exp_2}.

\begin{figure}[tbh]
	\centering
    \subfigure{	
    \includegraphics[width=\linewidth]{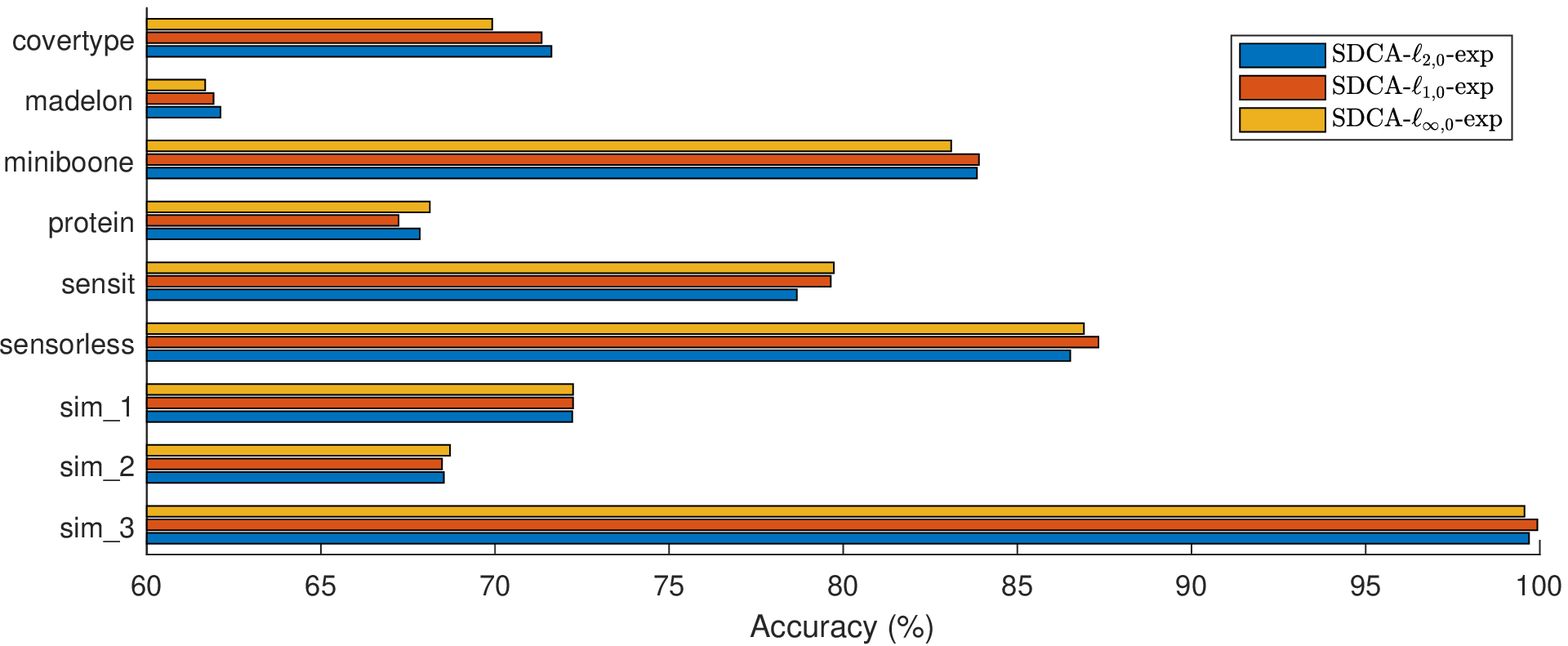} 	}
     \subfigure{	
    \includegraphics[width=\linewidth]{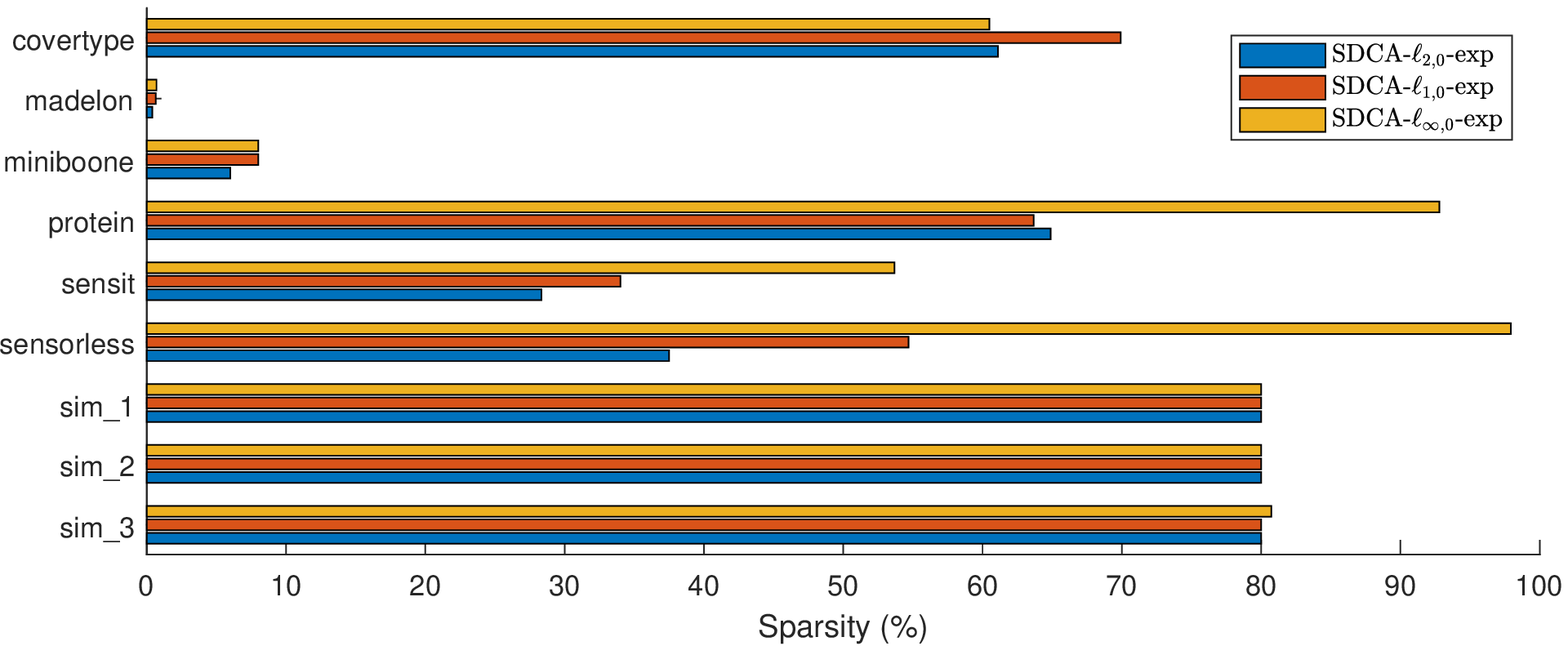} 	}
    \subfigure{	
    \includegraphics[width=\linewidth]{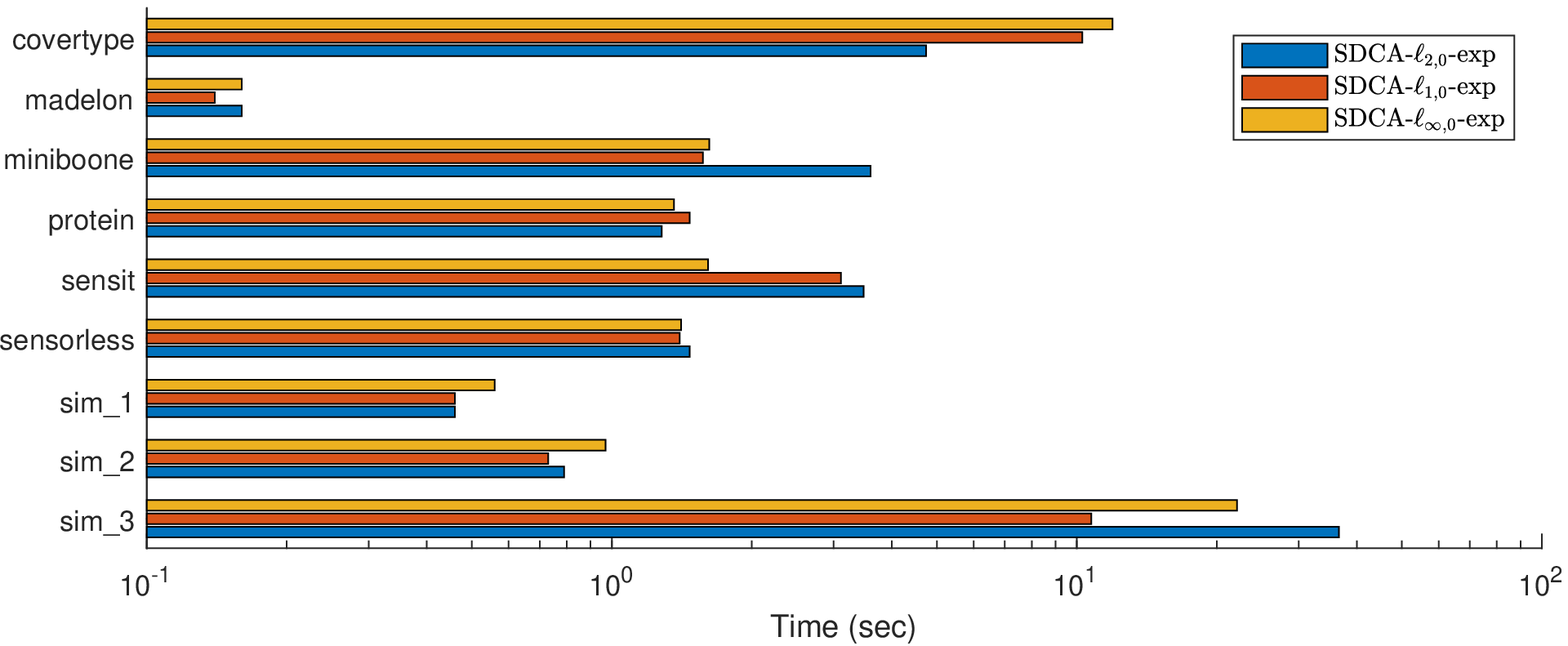} 	
}
\caption{Comparative results between \texttt{SDCA-$\ell_{1, 0}$-exp}, \texttt{SDCA-$\ell_{2, 0}$-exp} and \texttt{SDCA-$\ell_{\infty, 0}$-exp} (running time is plotted on a logarithmic scale).}
\label{fig:exp_2}
\end{figure}

\sloppy In term of classification accuracy, \texttt{SDCA-$\ell_{1,0}$-exp} and \texttt{SDCA-$\ell_{2,0}$-exp} are comparable and are slightly better than \texttt{SDCA-$\ell_{\infty, 0}$-exp}. \texttt{SDCA-$\ell_{1, 0}$-exp} produces similar results with \texttt{SDCA-$\ell_{2, 0}$-exp} on $6$ out of $9$ datasets, where the gap is lower than $0.3\%$ in classification accuracy. 
For \textit{protein}, \textit{sensorless} and \textit{sensit}, \texttt{SDCA-$\ell_{\infty, 0}$-exp} provides slightly better classification accuracy than \texttt{SDCA-$\ell_{1,0}$-exp} and \texttt{SDCA-$\ell_{2,0}$-exp}. This is due to the fact that \texttt{SDCA-$\ell_{\infty, 0}$-exp} selects much more variables than the two others.

\sloppy As for the sparsity of solution, \texttt{SDCA-$\ell_{2,0}$-exp} is the best on $8$ out of $9$ datasets (except for \textit{protein}).
\texttt{SDCA-$\ell_{1, 0}$-exp} selects moderately more variables than \texttt{SDCA-$\ell_{2, 0}$-exp}, from $5.67\%$ to $17.19\%$.
In contrast to \texttt{SDCA-$\ell_{2,0}$-exp}, \texttt{SDCA-$\ell_{\infty,0}$-exp} suppresses less variables than \texttt{SDCA-$\ell_{1,0}$-exp} and \texttt{SDCA-$\ell_{2,0}$-exp} on all datasets, except \textit{covertype}. Especially, on dataset \textit{sensorless}, \texttt{SDCA-$\ell_{\infty,0}$-exp} selects $60.42\%$ (resp. $43.23\%$) more variables than \texttt{SDCA-$\ell_{2,0}$-exp} (resp. \texttt{SDCA-$\ell_{1,0}$-exp}).

\sloppy In term of running time, \texttt{SDCA-$\ell_{1, 0}$-exp} is the fastest and \texttt{SDCA-$\ell_{2, 0}$-exp} is the slowest among the three algorithms. \texttt{SDCA-$\ell_{1, 0}$-exp} is up to $3.4$ time faster than \texttt{SDCA-$\ell_{2, 0}$-exp} and $2.06$ times faster than \texttt{SDCA-$\ell_{\infty, 0}$-exp}.

\sloppy Overall, \texttt{SDCA-$\ell_{1,0}$-exp} and \texttt{SDCA-$\ell_{2,0}$-exp} provide comparable results and realize a better trade-off between classification and sparsity of solution than \texttt{SDCA-$\ell_{\infty,0}$-exp}.


\subsection{Experiment 3} \label{exp_3}

\sloppy In this experiment, to study the effect of the approximation functions (capped-$\ell_1$ and exponential approximation), we compare two algorithms: \texttt{SDCA-$\ell_{2, 0}$-exp} and \texttt{SDCA-$\ell_{2, 0}$-cap$\ell_1$}. It is worth to note that capped-$\ell_1$  function is nonsmooth, hence the resulting approximate problem is a nonsmooth (and nonconvex) problem. The results are reported in Figure~\ref{fig:exp_3} and Table~\ref{tbl.experiment}. 

For \textit{sensit}, \textit{madelon}, \textit{sim\_1}, \textit{sim\_2} dataset, both algorithms have similar performance in all three criteria. 
The differences in terms of accuracy are negligible ($< 0.1\%$), while the gaps of sparsity and running time are mostly the same.

For \textit{sim\_3} and \textit{miniboone} dataset, both algorithms choose the same number of features.
However, \texttt{SDCA-$\ell_{2, 0}$-cap$\ell_1$} is faster than \texttt{SDCA-$\ell_{2, 0}$-exp} (by $41\%$ and $67\%$ respectively), while \texttt{SDCA-$\ell_{2, 0}$-exp} gives better (or similar) result in terms of classification accuracy.

For \textit{covertype},
\textit{sensorless} and \textit{protein} dataset,
\texttt{SDCA-$\ell_{2,0}$-exp} provides better results than \texttt{SDCA-$\ell_{2,0}$-cap$\ell_1$}.
\texttt{SDCA-$\ell_{2, 0}$-exp} furnishes results with higher classification accuracy in $2$ out of $3$ cases (\textit{covertype} and \textit{sensorless}) while having lower lower sparsity in $2$ out of $3$ cases (\textit{protein} and \textit{sensorless}).
In terms of running time, \texttt{SDCA-$\ell_{2, 0}$-exp} is faster than \texttt{SDCA-$\ell_{2, 0}$-cap$\ell_1$} by at least $1.5$ times.

Overall, \texttt{SDCA-$\ell_{2,0}$-exp} clearly shows better results \texttt{SDCA-$\ell_{2,0}$-cap$\ell_1$} in three criteria.

\begin{figure}[tbh]
	\centering
    \subfigure{	
    \includegraphics[width=\linewidth]{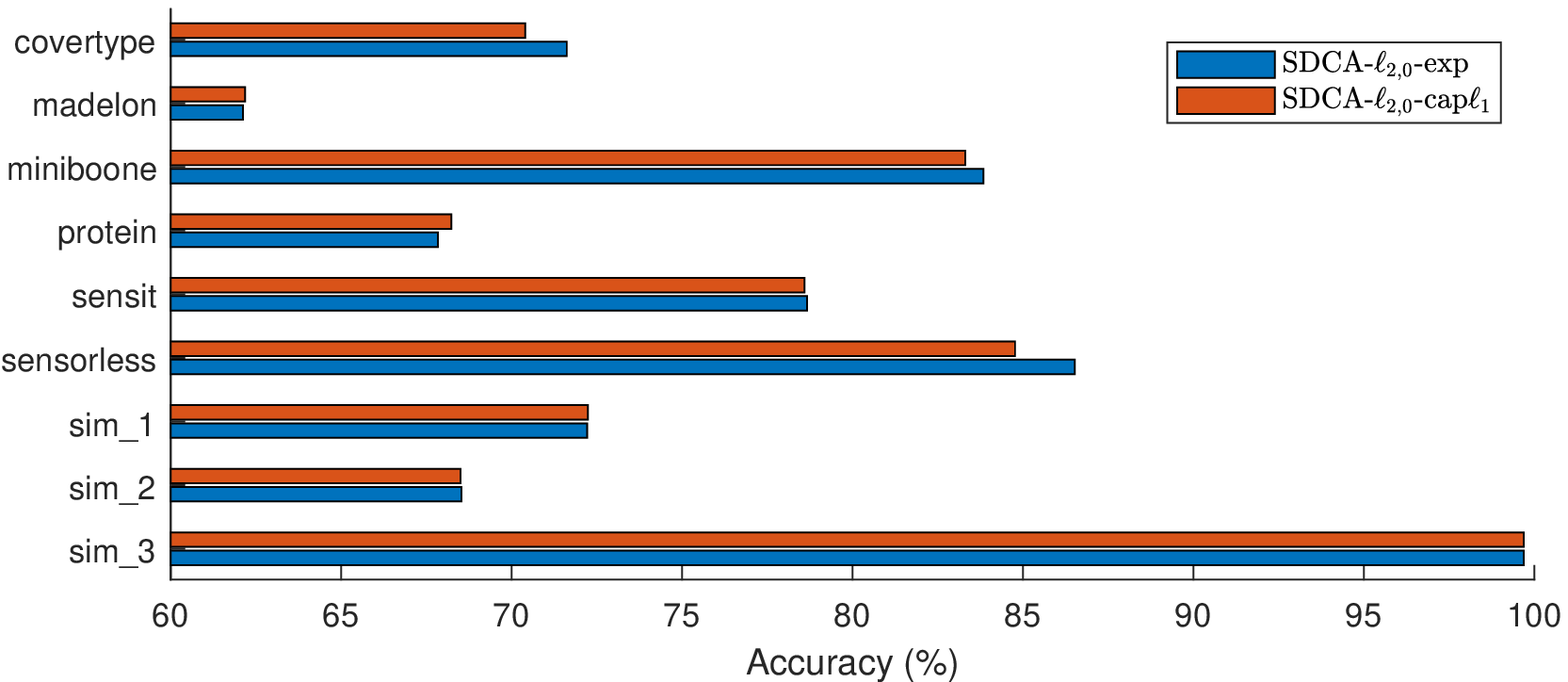} 	}
     \subfigure{	
    \includegraphics[width=\linewidth]{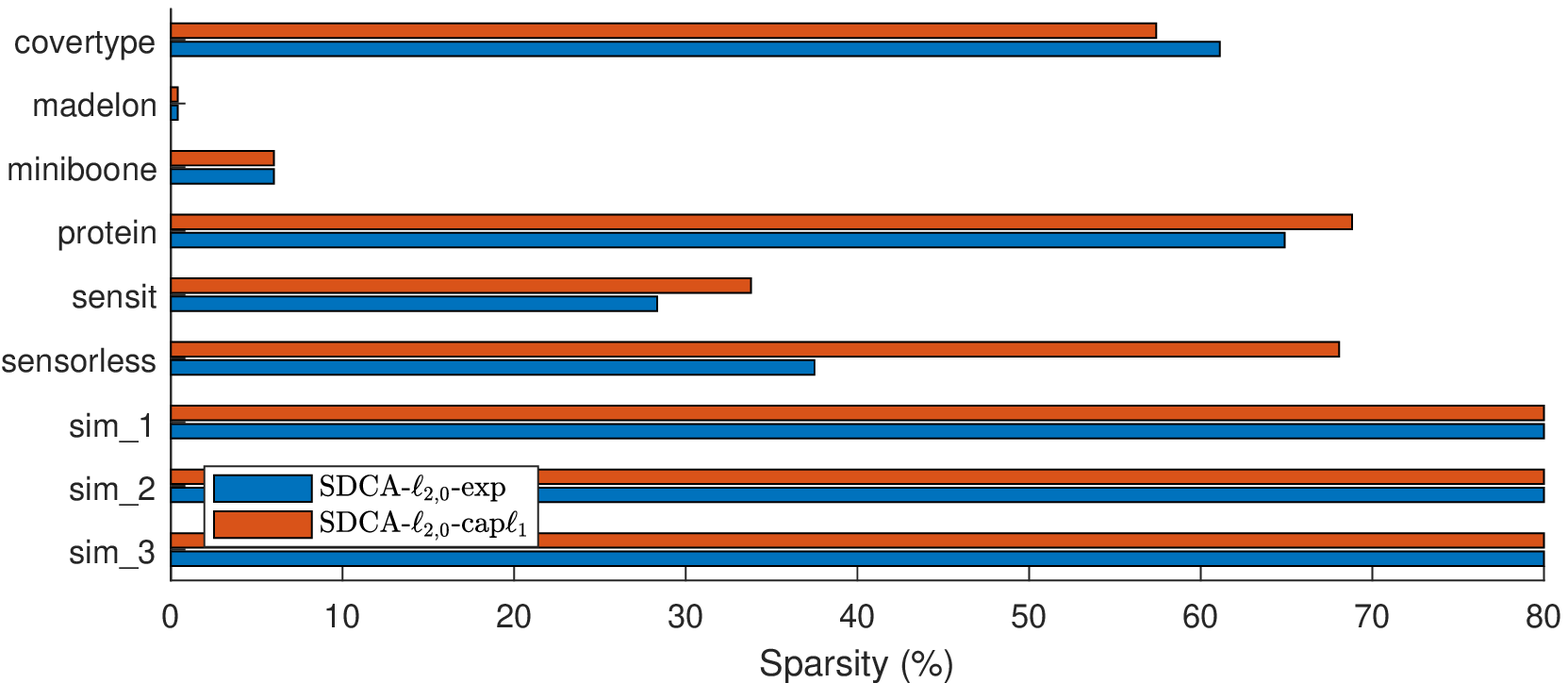} 	}
    \subfigure{	
    \includegraphics[width=\linewidth]{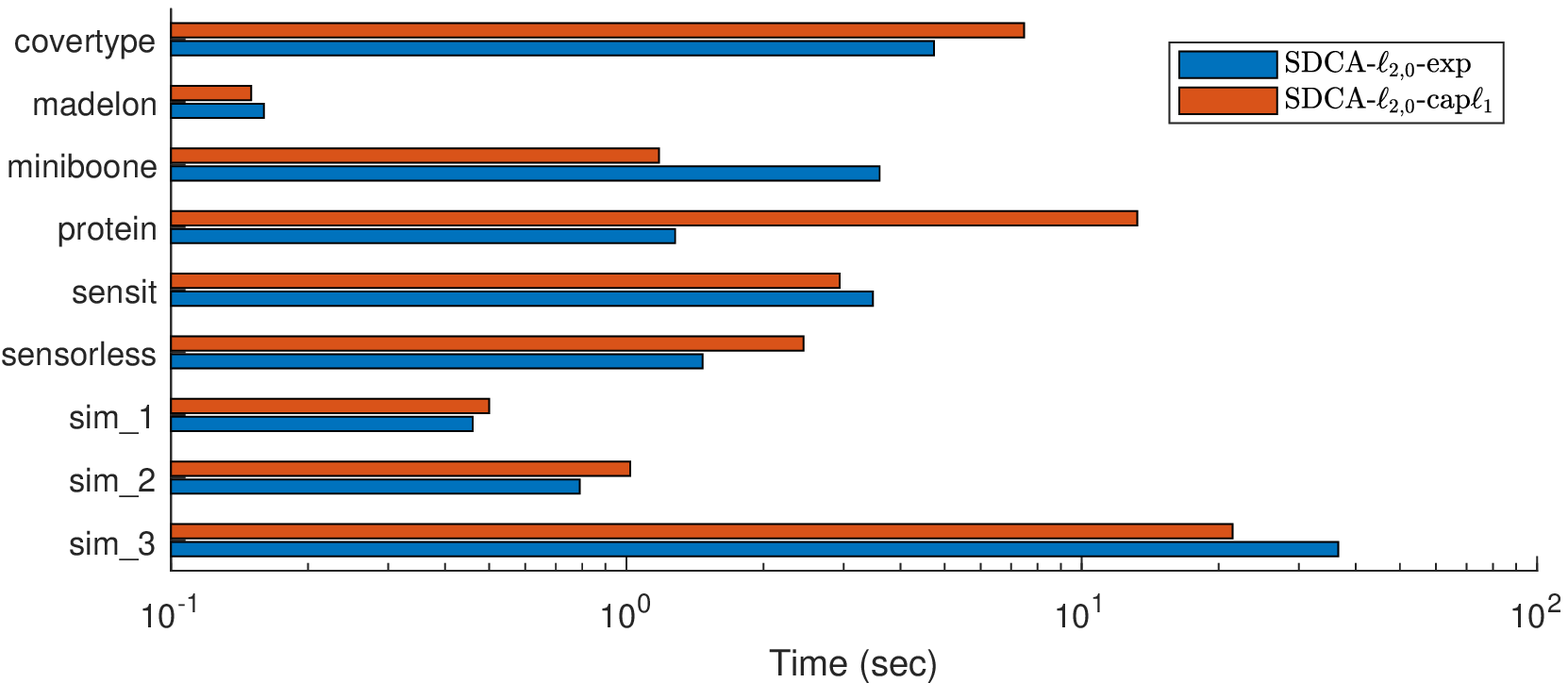} 	
}
\caption{Comparative results between \texttt{SDCA-$\ell_{2, 0}$-exp} and  \texttt{SDCA-$\ell_{2, 0}$-cap$\ell_1$} (running time is plotted on a logarithmic scale).}
\label{fig:exp_3}
\end{figure}


\begin{longtable}[!tbh]{llrrrrrr}
\caption{Comparative results on both synthetic and real datasets. \newline Bold values correspond to best results for each dataset. $n$, $d$ and $Q$ is the number of instances, the number of variables and the number of classes respectively.} 
\label{tbl.experiment}  
\\
\toprule
\multirow{2}{*}{Dataset}  & \multicolumn{1}{c}{\multirow{2}{*}{Algorithm}} & \multicolumn{2}{c}{Accuracy (\%)}    & \multicolumn{2}{c}{Time (s)}          & \multicolumn{2}{c}{Sparsity (\%)}    \\
\cmidrule(lr){3-4}\cmidrule(lr){5-6}\cmidrule(lr){7-8}
& \multicolumn{1}{r}{} & \multicolumn{1}{r}{Mean} & STD  & \multicolumn{1}{r}{Mean} & STD    & \multicolumn{1}{r}{Mean} & STD  \\
\midrule
\textit{covertype} & \texttt{SDCA-$\ell_{2, 0}$-exp} & 71.62 & 0.05 & \textbf{4.74} & \textbf{0.07} & 61.11 & 3.21 \\
$n = 581,012$ & \texttt{SDCA-$\ell_{1, 0}$-exp} & 71.34 & 0.07 & 10.27 & 1.25 & 69.91 & 1.77 \\
$d = 54$ & \texttt{SDCA-$\ell_{\infty, 0}$-exp} & 69.92 & 0.38 & 11.93 & 0.88 & 60.49 & 1.51 \\
$Q = 7$ & \texttt{SDCA-$\ell_{2, 0}$-cap$\ell_1$} & 70.40 & 0.03 & 7.47 & 5.69 & 57.41 & 1.85 \\
 & \texttt{SDCA-$\ell_{1, 0}$-cap$\ell_1$} & 68.60 & 0.29 & 8.98 & 2.03 & \textbf{25.93} & \textbf{0.00} \\
 & \texttt{SDCA-$\ell_{\infty, 0}$-cap$\ell_1$} & 70.16 & 0.03 & 14.80 & 3.63 & 56.79 & 3.85 \\
 & \texttt{DCA-$\ell_{2, 0}$-exp} & 72.15 & 0.08 & 92.73 & 0.51 & 64.81 & 1.51 \\
 & \texttt{DCA-$\ell_{1, 0}$-exp} & \textbf{72.28} & \textbf{0.07} & 57.93 & 2.87 & 73.61 & 0.93 \\
 & \texttt{DCA-$\ell_{\infty, 0}$-exp} & 69.39 & 0.10 & 57.22 & 5.36 & 42.13 & 0.93 \\
 & \texttt{DCA-$\ell_{2, 0}$-cap$\ell_1$} & 70.40 & 0.03 & 61.15 & 3.34 & 57.41 & 1.85 \\
 & \texttt{DCA-$\ell_{1, 0}$-cap$\ell_1$} & 69.41 & 0.39 & 37.20 & 1.23 & 69.14 & 1.07 \\
 & \texttt{DCA-$\ell_{\infty, 0}$-cap$\ell_1$} & 72.09 & 0.13 & 19.99 & 0.10 & 49.38 & 5.35 \\
 & \texttt{SPGD-$\ell_{2,1}$} & 66.97 & 0.51 & 60.59 & 7.09 & 100.00 & 0.00 \\
 & \texttt{msgl} & 71.22 & 0.02 & 525.49 & 1.10 & 68.52 & 0.00 \\
& & & & & & & \\
\textit{madelon} & \texttt{SDCA-$\ell_{2, 0}$-exp} & \textbf{62.12} & \textbf{1.00} & 0.16 & 0.02 & \textbf{0.40} & \textbf{0.12} \\
$n = 2,600 $ & \texttt{SDCA-$\ell_{1, 0}$-exp} & 61.92 & 0.80 & \textbf{0.14} & \textbf{0.03} & 0.65 & 0.10 \\
$d = 500$ & \texttt{SDCA-$\ell_{\infty, 0}$-exp} & 61.68 & 1.05 & 0.16 & 0.01 & 0.70 & 1.47 \\
$Q = 2$ & \texttt{SDCA-$\ell_{2, 0}$-cap$\ell_1$} & 62.18 & 1.35 & 0.15 & 0.09 & 0.40 & 0.00 \\
 & \texttt{SDCA-$\ell_{1, 0}$-cap$\ell_1$} & 61.73 & 1.26 & 0.16 & 0.02 & 1.53 & 0.12 \\
 & \texttt{SDCA-$\ell_{\infty, 0}$-cap$\ell_1$} & 61.99 & 1.06 & 0.16 & 0.31 & 10.60 & 0.20 \\
 & \texttt{DCA-$\ell_{2, 0}$-exp} & 61.54 & 0.79 & 0.29 & 0.27 & 0.85 & 0.19 \\
 & \texttt{DCA-$\ell_{1, 0}$-exp} & 61.83 & 1.12 & 0.32 & 0.02 & 0.55 & 0.10 \\
 & \texttt{DCA-$\ell_{\infty, 0}$-exp} & 61.88 & 1.03 & 2.17 & 0.01 & 4.65 & 0.25 \\
 & \texttt{DCA-$\ell_{2, 0}$-cap$\ell_1$} & 61.28 & 2.23 & 0.21 & 0.00 & 0.93 & 0.12 \\
 & \texttt{DCA-$\ell_{1, 0}$-cap$\ell_1$} & 61.54 & 1.57 & 0.41 & 0.00 & 1.07 & 0.31 \\
 & \texttt{DCA-$\ell_{\infty, 0}$-cap$\ell_1$} & 60.58 & 1.07 & 0.35 & 0.01 & 2.73 & 0.23 \\
 & \texttt{SPGD-$\ell_{2,1}$} & 61.79 & 0.80 & 1.07 & 0.03 & 1.00 & 0.20 \\
 & \texttt{msgl} & 60.48 & 2.37 & 23.92 & 0.12 & 0.67 & 0.00 \\ 
& & & & & & & \\
\textit{miniboone} & \texttt{SDCA-$\ell_{2, 0}$-exp} & 83.84 & 0.08 & 3.60 & 0.04 & 6.00 & 0.00 \\
$n = 130,065$ & \texttt{SDCA-$\ell_{1, 0}$-exp} & 83.90 & 0.10 & 1.57 & 0.06 & 8.00 & 0.00 \\
$d = 50$ & \texttt{SDCA-$\ell_{\infty, 0}$-exp} & 83.10 & 0.22 & 1.62 & 0.04 & 8.00 & 0.00 \\
$Q = 2$ & \texttt{SDCA-$\ell_{2, 0}$-cap$\ell_1$} & 83.31 & 0.15 & \textbf{1.18} & \textbf{0.01} & 6.00 & 0.00 \\
 & \texttt{SDCA-$\ell_{1, 0}$-cap$\ell_1$} & 82.50 & 0.06 & 2.96 & 0.19 & 8.00 & 0.00 \\
 & \texttt{SDCA-$\ell_{\infty, 0}$-cap$\ell_1$} & 83.77 & 0.10 & 4.22 & 0.28 & 16.00 & 4.00 \\
 & \texttt{DCA-$\ell_{2, 0}$-exp} & 83.93 & 0.12 & 2.49 & 0.31 & 6.00 & 0.00 \\
 & \texttt{DCA-$\ell_{1, 0}$-exp} & \textbf{84.19} & \textbf{0.15} & 9.42 & 0.09 & 8.00 & 0.00 \\
 & \texttt{DCA-$\ell_{\infty, 0}$-exp} & 81.54 & 0.12 & 9.81 & 3.45 & 8.00 & 0.00 \\
 & \texttt{DCA-$\ell_{2, 0}$-cap$\ell_1$} & 83.74 & 0.07 & 7.04 & 0.01 & 6.00 & 0.00 \\
 & \texttt{DCA-$\ell_{1, 0}$-cap$\ell_1$} & 83.11 & 0.05 & 7.54 & 0.00 & \textbf{4.00} & \textbf{0.00} \\
 & \texttt{DCA-$\ell_{\infty, 0}$-cap$\ell_1$} & 82.81 & 0.09 & 7.14 & 0.00 & 15.33 & 1.15 \\
 & \texttt{SPGD-$\ell_{2,1}$} & 83.86 & 0.13 & 8.77 & 0.41 & 11.00 & 1.15 \\
 & \texttt{msgl} & 81.99 & 0.21 & 121.17 & 4.30 & 10.00 & 0.00 \\
 & & & & & & & \\
\textit{protein} & \texttt{SDCA-$\ell_{2, 0}$-exp} & 67.84 & 1.11 & 1.28 & 0.06 & 64.89 & 1.95 \\
$n = 24,387$ & \texttt{SDCA-$\ell_{1, 0}$-exp} & 67.23 & 0.90 & 1.47 & 0.02 & 63.67 & 2.39 \\
$d = 357$ & \texttt{SDCA-$\ell_{\infty, 0}$-exp} & 68.13 & 0.57 & 1.36 & 0.06 & 92.79 & 0.86 \\
$Q = 3$ & \texttt{SDCA-$\ell_{2, 0}$-cap$\ell_1$} & 66.41 & 1.12 & \textbf{1.13} & \textbf{0.12} & \textbf{22.64} & \textbf{0.47} \\
 & \texttt{SDCA-$\ell_{1, 0}$-cap$\ell_1$} & 67.25 & 1.24 & 1.33 & 0.14 & 65.73 & 1.09 \\
 & \texttt{SDCA-$\ell_{\infty, 0}$-cap$\ell_1$} & \textbf{68.19} & \textbf{1.06} & \textbf{1.13} & \textbf{0.10} & 77.47 & 0.42 \\
 & \texttt{DCA-$\ell_{2, 0}$-exp} & 67.23 & 0.75 & 2.59 & 0.02 & 42.56 & 1.66 \\
 & \texttt{DCA-$\ell_{1, 0}$-exp} & 66.19 & 0.96 & 3.77 & 0.41 & 33.36 & 1.87 \\
 & \texttt{DCA-$\ell_{\infty, 0}$-exp} & 66.93 & 0.75 & 13.53 & 2.12 & 54.21 & 0.61 \\
 & \texttt{DCA-$\ell_{2, 0}$-cap$\ell_1$} & 67.04 & 0.72 & 3.35 & 0.00 & 50.47 & 1.27 \\
 & \texttt{DCA-$\ell_{1, 0}$-cap$\ell_1$} & 67.89 & 0.60 & 3.43 & 0.00 & 79.68 & 0.58 \\
 & \texttt{DCA-$\ell_{\infty, 0}$-cap$\ell_1$} & 66.90 & 0.84 & 3.66 & 1.04 & 58.43 & 1.46 \\
 & \texttt{SPGD-$\ell_{2,1}$} & 66.59 & 1.82 & 11.73 & 2.80 & 92.70 & 2.50 \\
 & \texttt{msgl} & 67.34 & 0.48 & 5.59 & 0.36 & 47.15 & 1.32 \\
 & & & & & & & \\
\textit{sensit} & \texttt{SDCA-$\ell_{2, 0}$-exp} & 78.67 & 0.11 & 3.48 & 0.21 & 28.33 & 8.50 \\
$n = 98,528$ & \texttt{SDCA-$\ell_{1, 0}$-exp} & 79.64 & 0.22 & 3.11 & 0.96 & 34.00 & 17.35 \\
$d = 100$ & \texttt{SDCA-$\ell_{\infty, 0}$-exp} & 79.73 & 0.28 & \textbf{1.61} & \textbf{0.07} & 53.67 & 6.81 \\
$Q = 3$ & \texttt{SDCA-$\ell_{2, 0}$-cap$\ell_1$} & 78.59 & 0.08 & 2.94 & 0.17 & 33.80 & 5.31 \\
 & \texttt{SDCA-$\ell_{1, 0}$-cap$\ell_1$} & 79.71 & 0.23 & 2.94 & 2.12 & 100.00 & 0.00 \\
 & \texttt{SDCA-$\ell_{\infty, 0}$-cap$\ell_1$} & 78.83 & 0.24 & 2.91 & 0.20 & 35.00 & 2.74 \\
 & \texttt{DCA-$\ell_{2, 0}$-exp} & \textbf{79.84} & \textbf{0.11} & 27.97 & 0.80 & 19.25 & 0.50 \\
 & \texttt{DCA-$\ell_{1, 0}$-exp} & 79.65 & 0.21 & 18.31 & 4.90 & \textbf{17.50} & \textbf{0.58} \\
 & \texttt{DCA-$\ell_{\infty, 0}$-exp} & 79.16 & 0.17 & 42.91 & 5.24 & 91.50 & 2.38 \\
 & \texttt{DCA-$\ell_{2, 0}$-cap$\ell_1$} & 78.92 & 0.15 & 26.36 & 2.22 & 56.67 & 1.53 \\
 & \texttt{DCA-$\ell_{1, 0}$-cap$\ell_1$} & 78.91 & 0.38 & 27.05 & 2.80 & 57.33 & 0.58 \\
 & \texttt{DCA-$\ell_{\infty, 0}$-cap$\ell_1$} & 79.20 & 0.17 & 35.78 & 2.69 & 91.67 & 7.23 \\
 & \texttt{SPGD-$\ell_{2,1}$} & 79.52 & 0.27 & 22.44 & 2.41 & 27.00 & 1.00 \\
 & \texttt{msgl} & 79.02 & 0.13 & 11.16 & 0.53 & 23.00 & 0.00 \\
 & & & & & & & \\
\textit{sensorless} & \texttt{SDCA-$\ell_{2, 0}$-exp} & 86.52 & 0.78 & 1.47 & 0.16 & 37.50 & 5.10 \\
$n = 58,509$ & \texttt{SDCA-$\ell_{1, 0}$-exp} & 87.33 & 0.27 & 1.40 & 0.09 & 54.69 & 10.67 \\
$d = 48$ & \texttt{SDCA-$\ell_{\infty, 0}$-exp} & 86.91 & 0.19 & 1.41 & 0.38 & 97.92 & 2.08 \\
$Q = 11$ & \texttt{SDCA-$\ell_{2, 0}$-cap$\ell_1$} & 84.77 & 0.08 & 2.45 & 0.13 & 68.06 & 1.20 \\
 & \texttt{SDCA-$\ell_{1, 0}$-cap$\ell_1$} & 82.89 & 0.30 & 2.69 & 0.62 & 72.92 & 2.08 \\
 & \texttt{SDCA-$\ell_{\infty, 0}$-cap$\ell_1$} & 87.12 & 0.72 & \textbf{1.36} & \textbf{0.09} & \textbf{25.69} & \textbf{1.20} \\
 & \texttt{DCA-$\ell_{2, 0}$-exp} & 90.00 & 0.31 & 15.96 & 0.65 & 32.81 & 1.04 \\
 & \texttt{DCA-$\ell_{1, 0}$-exp} & 89.11 & 0.18 & 16.28 & 0.97 & 31.25 & 0.00 \\
 & \texttt{DCA-$\ell_{\infty, 0}$-exp} & \textbf{90.76} & \textbf{0.14} & 18.99 & 0.81 & 100.00 & 0.00 \\
 & \texttt{DCA-$\ell_{2, 0}$-cap$\ell_1$} & 89.60 & 1.15 & 24.75 & 1.39 & 53.47 & 1.20 \\
 & \texttt{DCA-$\ell_{1, 0}$-cap$\ell_1$} & 88.87 & 1.04 & 16.28 & 0.44 & 47.92 & 0.80 \\
 & \texttt{DCA-$\ell_{\infty, 0}$-cap$\ell_1$} & 81.06 & 3.9 & 14.99 & 3.22 & 41.67 & 0.70 \\
 & \texttt{SPGD-$\ell_{2,1}$} & 86.07 & 1.39 & 8.16 & 1.05 & 88.89 & 2.41 \\
 & \texttt{msgl} & 85.06 & 0.31 & 199.00 & 41.75 & 50.00 & 0.00 \\ 
 & & & & & & & \\
\textit{sim\_1} & \texttt{SDCA-$\ell_{2, 0}$-exp} & 72.22 & 0.46 & 0.46 & 0.02 & \textbf{80.00} & \textbf{0.00} \\
$n = 100,000$ & \texttt{SDCA-$\ell_{1, 0}$-exp} & 72.24 & 0.43 & 0.46 & 0.03 & \textbf{80.00} & \textbf{0.00} \\
$d = 50$ & \texttt{SDCA-$\ell_{\infty, 0}$-exp} & 72.24 & 0.47 & 0.56 & 0.06 & \textbf{80.00} & \textbf{0.00} \\
$Q = 4$ & \texttt{SDCA-$\ell_{2, 0}$-cap$\ell_1$} & 72.24 & 0.52 & 0.50 & 0.04 & \textbf{80.00} & \textbf{0.00} \\
 & \texttt{SDCA-$\ell_{1, 0}$-cap$\ell_1$} & 72.24 & 0.58 & 0.42 & 0.06 & \textbf{80.00} & \textbf{0.00} \\
 & \texttt{SDCA-$\ell_{\infty, 0}$-cap$\ell_1$} & 72.21 & 0.58 & 0.51 & 0.07 & \textbf{80.00} & \textbf{0.00} \\
 & \texttt{DCA-$\ell_{2, 0}$-exp} & 72.22 & 0.40 & 2.34 & 0.05 & \textbf{80.00} & \textbf{0.00} \\
 & \texttt{DCA-$\ell_{1, 0}$-exp} & 72.25 & 0.38 & 0.26 & 0.01 & \textbf{80.00} & \textbf{0.00} \\
 & \texttt{DCA-$\ell_{\infty, 0}$-exp} & 72.22 & 0.40 & 9.79 & 0.10 & \textbf{80.00} & \textbf{0.00} \\
 & \texttt{DCA-$\ell_{2, 0}$-cap$\ell_1$} & 72.25 & 0.52 & 0.32 & 0.00 & \textbf{80.00} & \textbf{0.00} \\
 & \texttt{DCA-$\ell_{1, 0}$-cap$\ell_1$} & 72.24 & 0.52 & \textbf{0.30} & \textbf{0.00} & \textbf{80.00} & \textbf{0.00} \\
 & \texttt{DCA-$\ell_{\infty, 0}$-cap$\ell_1$} & 72.24 & 0.51 & 3.00 & 0.00 & \textbf{80.00} & \textbf{0.00} \\
 & \texttt{SPGD-$\ell_{2,1}$} & 71.48 & 0.81 & 7.16 & 0.91 & 83.50 & 2.52 \\
 & \texttt{msgl} & \textbf{72.33} & 0.18 & 214.83 & 25.40 & 82.00 & 0.00 \\ 
 & & & & & & & \\
\textit{sim\_2} & \texttt{SDCA-$\ell_{2, 0}$-exp} & 68.53 & 0.29 & 0.79 & 0.00 & \textbf{80.00} & \textbf{0.00} \\
$n = 150,000$ & \texttt{SDCA-$\ell_{1, 0}$-exp} & 68.48 & 0.34 & 0.73 & 0.16 & \textbf{80.00} & \textbf{0.00} \\
$d = 50$ & \texttt{SDCA-$\ell_{\infty, 0}$-exp} & \textbf{68.71} & \textbf{0.23} & 0.97 & 0.12 & \textbf{80.00} & \textbf{0.00} \\
$Q = 3$ & \texttt{SDCA-$\ell_{2, 0}$-cap$\ell_1$} & 68.50 & 0.29 & 1.02 & 0.14 & \textbf{80.00} & \textbf{0.00} \\
 & \texttt{SDCA-$\ell_{1, 0}$-cap$\ell_1$} & 67.42 & 0.40 & \textbf{0.71} & \textbf{0.23} & \textbf{80.00} & \textbf{0.00} \\
 & \texttt{SDCA-$\ell_{\infty, 0}$-cap$\ell_1$} & 68.38 & 0.28 & 1.40 & 0.18 & \textbf{80.00} & \textbf{0.00} \\
 & \texttt{DCA-$\ell_{2, 0}$-exp} & 68.55 & 0.22 & 1.14 & 0.26 & \textbf{80.00} & \textbf{0.00} \\
 & \texttt{DCA-$\ell_{1, 0}$-exp} & 68.31 & 0.23 & 13.51 & 1.93 & \textbf{80.00} & \textbf{0.00} \\
 & \texttt{DCA-$\ell_{\infty, 0}$-exp} & \textbf{68.71} & \textbf{0.18} & 2.75 & 2.80 & \textbf{80.00} & \textbf{0.00} \\
 & \texttt{DCA-$\ell_{2, 0}$-cap$\ell_1$} & 67.70 & 0.31 & 4.29 & 0.03 & \textbf{80.00} & \textbf{0.00} \\
 & \texttt{DCA-$\ell_{1, 0}$-cap$\ell_1$} & 68.43 & 0.24 & 0.93 & 0.16 & \textbf{80.00} & \textbf{0.00} \\
 & \texttt{DCA-$\ell_{\infty, 0}$-cap$\ell_1$} & 67.49 & 0.35 & 0.69 & 0.10 & \textbf{80.00} & \textbf{0.00} \\
 & \texttt{SPGD-$\ell_{2,1}$} & 67.62 & 0.48 & 7.77 & 0.28 & 82.00 & 0.00 \\
 & \texttt{msgl} & 68.42 & 0.03 & 367.29 & 53.52 & 82.00 & 0.00 \\ 
 & & & & & & & \\
\textit{sim\_3} & \texttt{SDCA-$\ell_{2, 0}$-exp} & 99.69 & 0.04 & 36.61 & 1.48 & \textbf{80.00} & \textbf{0.00} \\
$n = 250,000$ & \texttt{SDCA-$\ell_{1, 0}$-exp} & \textbf{99.93} & \textbf{0.01} & \textbf{10.74} & \textbf{0.42} & \textbf{80.00} & \textbf{0.00} \\
$d = 500$ & \texttt{SDCA-$\ell_{\infty, 0}$-exp} & 99.56 & 0.07 & 22.11 & 3.43 & 80.73 & 0.64 \\
$Q = 4$ & \texttt{SDCA-$\ell_{2, 0}$-cap$\ell_1$} & 99.69 & 0.01 & 21.45 & 0.93 & \textbf{80.00} & \textbf{0.00} \\
 & \texttt{SDCA-$\ell_{1, 0}$-cap$\ell_1$} & 99.00 & 0.01 & 23.10 & 0.12 & \textbf{80.00} & \textbf{0.00} \\
 & \texttt{SDCA-$\ell_{\infty, 0}$-cap$\ell_1$} & 99.67 & 0.01 & 21.05 & 1.06 & \textbf{80.00} & \textbf{0.00} \\
 & \texttt{DCA-$\ell_{2, 0}$-exp} & 99.88 & 0.02 & 249.74 & 10.73 & \textbf{80.00} & \textbf{0.00} \\
 & \texttt{DCA-$\ell_{1, 0}$-exp} & 99.88 & 0.02 & 202.67 & 33.27 & \textbf{80.00} & \textbf{0.00} \\
 & \texttt{DCA-$\ell_{\infty, 0}$-exp} & 97.74 & 2.05 & 431.13 & 26.47 & \textbf{80.00} & \textbf{0.00} \\
 & \texttt{DCA-$\ell_{2, 0}$-cap$\ell_1$} & 99.92 & 0.01 & 178.89 & 7.83 & \textbf{80.00} & \textbf{0.00} \\
 & \texttt{DCA-$\ell_{1, 0}$-cap$\ell_1$} & 99.87 & 0.01 & 270.69 & 17.64 & \textbf{80.00} & \textbf{0.00} \\
 & \texttt{DCA-$\ell_{\infty, 0}$-cap$\ell_1$} & 99.85 & 0.03 & 24.40 & 4.29 & 80.40 & 0.40 \\
 & \texttt{SPGD-$\ell_{2,1}$} & 99.70 & 0.12 & 212.71 & 21.79 & \textbf{80.00} & \textbf{0.00} \\
 & \texttt{msgl} & \textbf{99.93} & \textbf{0.01} & 1581.44 & 14.76 & 80.20 & 0.00 \\
\hline
\end{longtable}

\section{Conclusion}\label{Conclusion}

We have proposed two novel DCA based algorithms, stochastic DCA and inexact stochastic DCA for minimizing a large sum of DC functions, with the aim of reducing the computation cost of DCA in large-scale setting. The sum structure of the objective function $F$ permits us to work separately on the component functions $F_i$. The stochastic DCA is then proposed to tackle problems with huge numbers of $F_i$ while the inexact stochastic DCA aims to address large-scale setting and big data. We have carefully studied the convergence properties of the proposed algorithms. It turns out that the convergence to a critical point of both stochastic DCA and inexact stochastic DCA is guaranteed with probability $1$. 
Furthermore, we have developed DCA and SDCA to group variables selection in multi-class logistic regression, an important problem in machine learning. By using a suitable DC decomposition of the objective function we have designed a DCA scheme in which all computations are explicit and inexpensive. Consequently SDCA is very inexpensive. Numerical results showed that, as expected, \texttt{SDCA-$\ell_{2,0}$-exp} reduces considerably the running time of \texttt{DCA-$\ell_{2,0}$-exp} while achieving equivalent classification accuracy. Moreover, \texttt{SDCA-$\ell_{2,0}$-exp} outperforms the two related algorithms \texttt{msgl} and \texttt{SPGD-$\ell_{2,1}$}. We are convinced that SDCA is an efficient variant of DCA, especially for large-scale setting.

Continuing this research direction, in future works we will develop novel versions of DCA based algorithms (e.g. online/stochastic/approximate/like DCA) for other problems in order to accelerate the convergence of DCA and to deal with large-scale setting and big data.

\appendix

\section{Proof of Theorem \ref{theorem2}}\label{app}
To prove Theorem \ref{theorem2}, we will use the following lemma.
\begin{lemma}\label{lemma1}
Let $f:\mathbb{R}^d\to\mathbb{R}\cap\{+\infty\}$ be a $\rho$-convex function. For any $\epsilon\geq 0$ and any $v\in\partial_{\epsilon} f(x)$ with $x\in$ dom $f$, we have
\begin{equation*}
2\epsilon + f(y) \geq f(x) + \langle v, y - x\rangle + \frac{\rho}{4}\|y-x\|^2,\ \forall y\in\mathbb{R}^d.
\end{equation*}
\end{lemma}
\begin{proof}
Since $v\in\partial_{\epsilon} f(x)$, we have
\begin{equation*}
\epsilon + f(z) \geq f(x) + \langle v, z - x\rangle,\ \forall z\in\mathbb{R}^d.
\end{equation*}
Replacing $z$ with $x+t(y-x)$ in this inequality gives that
\begin{equation*}
\epsilon + f(x+t(y-x)) \geq f(x) + t\langle v, y - x\rangle,\ \forall y\in\mathbb{R}^d.
\end{equation*}
It follows from the $\rho$-convexity of $f$ that for $y\in\mathbb{R}^d$ and $t\in (0,1)$,
\begin{equation*}
tf(y) + (1-t)f(x) \geq f(x + t(y-x)) + \frac{\rho t(1-t)}{2}\|y-x\|^2.
\end{equation*}
Summing the two above inequalities gives us 
\begin{equation*}
\epsilon + tf(y) \geq tf(x) + t\langle v, y - x\rangle + \frac{\rho t(1-t)}{2}\|y-x\|^2.
\end{equation*}
Thus, the conclusion follows from this inequality with $t=1/2$.
\end{proof}
\begin{proof}(of Theorem \ref{theorem2})
a) Let $x^0_i$ be the copies of $x^0$. We set $x^{l+1}_i = x^{l+1}$ for all $i\in s_{l+1}$ and $x^{l+1}_j = x^l_j$ for $j\not\in s_{l+1}$. Set $\epsilon^0_i = \epsilon^0$ and $\epsilon^{l+1}_i=\epsilon^{l+1}$ if $i\in s_{l+1}$, $\epsilon^l_i$ otherwise. We then have $v^l_i\in\partial_{\epsilon^l_i}h_i(x_i^l)$ for $i=1,...,n$. Let $T_i^l$ be the function given by
\begin{align*}
T_i^l(x) = g_i(x)  - h_i(x_i^l) - \left\langle x - x_i^l, v^l_i \right\rangle + 2\epsilon^l_i.
\end{align*}
It follows from $v^l_i\in\partial_{\epsilon^l_i}h_i(x_i^l)$ that
\begin{equation*}
\epsilon^l_i + h_i(x) \geq h_i(x_i^l) + \left\langle x - x_i^l, v^l_i \right\rangle.
\end{equation*}
This implies $T_i^l(x) \geq F_i(x) + \epsilon^l_i \geq F_i(x)$ for all $l \geq 0$, $i=1,...,n$. We also observe that $x^{l+1}$ is an $\epsilon^l$-solution of the following convex problem
\begin{equation}\label{sub-problem2}
\min_{x} T^l(x):=\frac{1}{n}\sum_{i=1}^nT_i^l(x)
\end{equation}
Therefore
\begin{equation}\label{eq4}
\begin{aligned}
T^l(x^{l+1}) \leq T^l(x^l) +\epsilon^l& = T^{l-1}(x^l)+ \frac{1}{n}\sum_{i\in s_l}[T_i^l(x^l) - T_i^{l-1}(x^l)] +\epsilon^l\\
& = T^{l-1}(x^l)+ \frac{1}{n}\sum_{i\in s_l}[F_i(x^l) + 2\epsilon^l - T_i^{l-1}(x^l)] + \epsilon^l,
\end{aligned}
\end{equation}
where the second equality follows from $T_i^l(x^l) = F_i(x^l) + 2\epsilon^l$ for all $i\in s_l$.
Let $\mathcal{F}_l$ denote the $\sigma$-algebra generated by the entire history of ISDCA up to the iteration $l$, i.e., $\mathcal{F}_0 = \sigma(x^0,\epsilon^0)$ and $\mathcal{F}_l = \sigma(x^0,...,x^l,\epsilon^0,...,\epsilon^l, s_0,...,s_{l-1})$ for all $l \geq 1$. By taking the expectation of the inequality \eqref{eq4} conditioned on $\mathcal{F}_l$, we have 
\begin{equation*}
\mathbb{E}\left[T^l(x^{l+1})|\mathcal{F}_l\right] \leq T^{l-1}(x^l) - \frac{b}{n}\left[T^{l-1}(x^l) - F(x^l)\right] + \left(\frac{2b}{n}+1
\right)\epsilon^l.
\end{equation*}
Since $\sum_{l=0}^\infty\epsilon^l_i < +\infty$ with probability $1$, by applying the supermartingale convergence theorem \citep{neveu75,Bertsekas03} to the nonnegative sequences $\{T^{l-1}(x^l) - \alpha^*\}, \{\frac{b}{n}[T^{l-1}(x^l) - F(x^l)]\}$ and $\{(\frac{2b}{n}+1)\epsilon^l\}$, we   conclude that the sequence $\{T^{l-1}(x^l,y^l) - \alpha^*\}$ converges to $T^* -\alpha^*$ and 
\begin{equation}\label{eq728}
\sum_{l=1}^\infty\left[ T^{l-1}(x^l) - F(x^l)\right] < \infty,
\end{equation}
with probability $1$. Therefore $\{F(x^l)\}$ converges almost surely to $T^*$.

b) By $v_i^{l-1}\in\partial_{\epsilon^{l-1}_i}h_i(x_i^{l-1})$ and Lemma \ref{lemma1}, we have
\begin{equation*}\label{eq1}
2\epsilon^{l-1}_i + h_i(x)  \geq  h_i(x_i^{l-1}) + \langle x-x_i^{l-1}, v_i^{l-1} \rangle + \frac{\rho(h_i)}{4}\|x - x_i^{l-1}\|^2,\ \forall x\in\mathbb{R}^d.
\end{equation*}
This implies
\begin{equation}\label{eq3}
F_i(x)  \leq T_i^{l-1}(x) - \frac{\rho(h_i)}{4}\|x - x_i^{l-1}\|^2.
\end{equation}
From \eqref{eq4} and \eqref{eq3} with $x=x^l$, we have 
\begin{equation}\label{eq5}
T^l(x^{l+1}) \leq T^{l-1}(x^l) - \frac{1}{n}\sum_{i\in s_l}\frac{\rho(h_i)}{4}\|x - x_i^{l-1}\|^2 + \left(\frac{2b}{n}+1
\right)\epsilon^l.
\end{equation} 
Taking the expectation of the inequality \eqref{eq5} conditioned on $\mathcal{F}_l$, we obtain 
\begin{equation*}
\mathbb{E}\left[T^l(x^{l+1})|\mathcal{F}_l\right] \leq T^{l-1}(x^l) - \frac{b}{4n^2}\sum_{i=1}^n\rho(h_i)\|x^l-x_i^{l-1}\|^2 + \left(\frac{2b}{n}+1
\right)\epsilon^l.
\end{equation*}
Combining this and $\rho = \min_{i=1,...,n}\rho(h_i)  >0$ gives us
\begin{equation*}
\mathbb{E}\left[T^l(x^{l+1})|\mathcal{F}_l\right] \leq T^{l-1}(x^l) - \frac{b\rho}{4n^2}\sum_{i=1}^n\|x^l-x_i^{l-1}\|^2 + \left(\frac{2b}{n}+1
\right)\epsilon^l.
\end{equation*}
Applying the supermartingale convergence theorem to the nonnegative sequences $\{T^{l-1}(x^l) - \alpha^*\}, \{\frac{b}{4\rho n^2}\sum_{i=1}^n\|x^l-x_i^{l-1}\|^2\}$ and $\{(\frac{2b}{n}+1)\epsilon^l\}$, we get 
\begin{equation*}\label{eq7}
\sum_{l=1}^\infty\sum_{i=1}^n\|x^l-x_i^{l-1}\|^2 < \infty,
\end{equation*}
with probability $1$. In particular, for $i=1,...,n$, we have
\begin{equation}\label{eq8}
\sum_{l=1}^\infty\|x^l-x_i^{l-1}\|^2 < \infty,
\end{equation}
and hence $\lim_{l \rightarrow \infty}\|x^l-x_i^{l-1}\| =0$ almost surely. 

c) Assume that there exists a sub-sequence $\{x^{l_k}\}$ of $\{x^l\}$ such that $x^{l_k} \rightarrow x^*$ almost surely. From \eqref{eq8}, we have $\|x^{l_k+1}-x_i^{l_k}\| \rightarrow 0$ almost surely. Without loss of generality, we can suppose that the sub-sequence $v_i^{l_k} \rightarrow v^*_i$ almost surely. From the proof of (a), we have
\begin{equation*}
\frac{1}{n}\sum_{i=1}^n\epsilon^l_i \leq T^l(x^{l+1}) - F(x^{l+1}).
\end{equation*}
From this and \eqref{eq728} it follows that $\epsilon^l_i$ converges to $0$ as $l\to +\infty$ with probability $1$. Since $v_i^{l_k}\in\partial_{\epsilon_i^{l_k}} h_i(x_i^{l_k})$, $\epsilon^{l_k}_i\to 0$ with probability $1$, and by the closed property of the $\epsilon$-subdifferential mapping $\partial_{\epsilon_i^{l_k}} h_i$, we have $v^*_i\in\partial h_i(x^*)$. Since $x^{l_k+1}$  is a $\epsilon^{l_k}$-solution of the problem $\min_x T^{l_k}(x)$, we obtain
\begin{equation}\label{eq9}
T^{l_k}(x^{l_k+1}) \leq T^{l_k}(x) +\epsilon^{l_k}.
\end{equation}
Taking $k\rightarrow \infty$ gives us 
\begin{equation*}\label{eq10}
\limsup_{l_k\to +\infty}G(x^{l_k+1})  \leq  G(x)  - \langle x- x^*, v^*\rangle,\ \forall x\in \mathbb{R}^d,
\end{equation*}
with probability $1$, where $v^* = \frac{1}{n}\sum_{i=1}^nv^*_i\in \partial H(x^*)$ almost surely. It follows from this with $x = x^*$ that
\begin{equation*}
\limsup_{l_k\to +\infty}G(x^{l_k+1}) \leq  G(x^*),
\end{equation*}
almost surely. Combining this with the lower semi-continuity of $G$ gives us that
\begin{equation*}
\lim_{l_k\to +\infty}G(x^{l_k+1}) =  G(x^*),
\end{equation*}
almost surely. Thus, we have
\begin{equation*}
G(x^*)  \leq  G(x)  - \langle x- x^*, v^*\rangle,\ \forall x\in \mathbb{R}^d,
\end{equation*}
almost surely. This implies
\begin{equation}\label{eq11}
v^*\in \partial G(x^*),
\end{equation}
with probability one. Therefore,
\begin{equation}
v^*\in \partial G(x^*)\cap\partial H(x^*),
\end{equation}
with probability $1$. This implies that $x^*$ is a critical point of $F$ with probability $1$ and the proof is then complete. 
\end{proof}


\begin{thebibliography}{42}
\expandafter\ifx\csname natexlab\endcsname\relax\def\natexlab#1{#1}\fi
\providecommand{\bibinfo}[2]{#2}
\ifx\xfnm\relax \def\xfnm[#1]{\unskip,\space#1}\fi
\bibitem[{Allen-Zhu \& Yuan(2016)}]{Allen-Zhu16}
\bibinfo{author}{Allen-Zhu, Z.}, \& \bibinfo{author}{Yuan, Y.}
  (\bibinfo{year}{2016}).
\newblock \bibinfo{title}{Improved {SVRG} for non-strongly-convex or
  sum-of-non-convex objectives}.
\newblock In {\it \bibinfo{booktitle}{Proceedings of the 33rd International
  Conference on International Conference on Machine Learning - Volume 48}\/}
  ICML'16 (pp. \bibinfo{pages}{1080--1089}).
\bibitem[{Bagley et~al.(2001)Bagley, White \& Golomb}]{Bagley2001}
\bibinfo{author}{Bagley, S.~C.}, \bibinfo{author}{White, H.}, \&
  \bibinfo{author}{Golomb, B.~A.} (\bibinfo{year}{2001}).
\newblock \bibinfo{title}{Logistic regression in the medical literature:
  Standards for use and reporting, with particular attention to one medical
  domain}.
\newblock {\it \bibinfo{journal}{Journal of Clinical Epidemiology}\/},  {\it
  \bibinfo{volume}{54}\/}, \bibinfo{pages}{979--985}.
\bibitem[{Bertsekas et~al.(2003)Bertsekas, Nedic \& Ozdaglar}]{Bertsekas03}
\bibinfo{author}{Bertsekas, D.}, \bibinfo{author}{Nedic, A.}, \&
  \bibinfo{author}{Ozdaglar, A.} (\bibinfo{year}{2003}).
\newblock {\it \bibinfo{title}{Convex analysis and optimization}\/}.
\newblock \bibinfo{publisher}{Athena Scientific}.
\bibitem[{Bertsekas(2010)}]{Bertsekas2010}
\bibinfo{author}{Bertsekas, D.~P.} (\bibinfo{year}{2010}).
\newblock {\it \bibinfo{title}{Incremental Gradient, Subgradient, and Proximal
  Methods for Convex Optimization: A Survey}\/}.
\newblock \bibinfo{type}{Technical Report} Laboratory for Information and
  Decision Systems, MIT, Cambridge, MA.
\bibitem[{Bertsekas(2011)}]{Berinc2011}
\bibinfo{author}{Bertsekas, D.~P.} (\bibinfo{year}{2011}).
\newblock \bibinfo{title}{Incremental proximal methods for large scale convex
  optimization}.
\newblock {\it \bibinfo{journal}{Mathematical Programming}\/},  {\it
  \bibinfo{volume}{129}\/}, \bibinfo{pages}{163--195}.
\bibitem[{Bottou(1998)}]{Bottou98}
\bibinfo{author}{Bottou, L.} (\bibinfo{year}{1998}).
\newblock \bibinfo{title}{On-line learning and stochastic approximations}.
\newblock In \bibinfo{editor}{D.~Saad} (Ed.), {\it \bibinfo{booktitle}{On-line
  Learning in Neural Networks}\/} (pp. \bibinfo{pages}{9--42}).
\newblock \bibinfo{address}{New York, NY, USA}: \bibinfo{publisher}{Cambridge
  University Press}.
\bibitem[{Boyd et~al.(1987)Boyd, Tolson \& Copes}]{Boyd1987}
\bibinfo{author}{Boyd, C.~R.}, \bibinfo{author}{Tolson, M.~A.}, \&
  \bibinfo{author}{Copes, W.~S.} (\bibinfo{year}{1987}).
\newblock \bibinfo{title}{Evaluating trauma care: The {{TRISS}} method.
  {{Trauma Score}} and the {{Injury Severity Score}}}.
\newblock {\it \bibinfo{journal}{The Journal of Trauma}\/},  {\it
  \bibinfo{volume}{27}\/}, \bibinfo{pages}{370--378}.
\bibitem[{Bradley \& Mangasarian(1998)}]{Bradley98}
\bibinfo{author}{Bradley, P.~S.}, \& \bibinfo{author}{Mangasarian, O.~L.}
  (\bibinfo{year}{1998}).
\newblock \bibinfo{title}{Feature selection via concave minimization and
  support vector machines}.
\newblock In {\it \bibinfo{booktitle}{Machine Learning Proceedings of the
  Fifteenth International Conference (ICML ’98)}\/} (pp.
  \bibinfo{pages}{82--90}).
\newblock \bibinfo{publisher}{Morgan Kaufmann}.
\bibitem[{Cox(1958)}]{Cox1958}
\bibinfo{author}{Cox, D.} (\bibinfo{year}{1958}).
\newblock \bibinfo{title}{The regression analysis of binary sequences (with
  discussion)}.
\newblock {\it \bibinfo{journal}{J Roy Stat Soc B}\/},  {\it
  \bibinfo{volume}{20}\/}, \bibinfo{pages}{215--242}.
\bibitem[{Defazio et~al.(2014{\natexlab{a}})Defazio, Bach \&
  Lacoste-Julien}]{Defazio2014a}
\bibinfo{author}{Defazio, A.}, \bibinfo{author}{Bach, F.}, \&
  \bibinfo{author}{Lacoste-Julien, S.} (\bibinfo{year}{2014}{\natexlab{a}}).
\newblock \bibinfo{title}{Saga: A fast incremental gradient method with support
  for non-strongly convex composite objectives}.
\newblock In {\it \bibinfo{booktitle}{Proceedings of Advances in Neural
  Information Processing Systems}\/}.
\bibitem[{Defazio et~al.(2014{\natexlab{b}})Defazio, Caetano \&
  Domke}]{Defazio2014b}
\bibinfo{author}{Defazio, A.}, \bibinfo{author}{Caetano, T.}, \&
  \bibinfo{author}{Domke, J.} (\bibinfo{year}{2014}{\natexlab{b}}).
\newblock \bibinfo{title}{Finito: A faster, permutable incremental gradient
  method for big data problems}.
\newblock In {\it \bibinfo{booktitle}{Proceedings of the $31^{st}$
  International Conference on Machine Learning}\/}.
\bibitem[{Genkin et~al.(2007)Genkin, Lewis \& Madigan}]{Genkin2007}
\bibinfo{author}{Genkin, A.}, \bibinfo{author}{Lewis, D.~D.}, \&
  \bibinfo{author}{Madigan, D.} (\bibinfo{year}{2007}).
\newblock \bibinfo{title}{Large-scale {{Bayesian}} logistic regression for text
  categorization}.
\newblock {\it \bibinfo{journal}{Technometrics}\/},  {\it
  \bibinfo{volume}{49}\/}, \bibinfo{pages}{291--304}.
\bibitem[{Healy \& Schruben(1991)}]{Healy1991}
\bibinfo{author}{Healy, K.}, \& \bibinfo{author}{Schruben, L.~W.}
  (\bibinfo{year}{1991}).
\newblock \bibinfo{title}{Retrospective simulation response optimization}.
\newblock In {\it \bibinfo{booktitle}{1991 Winter Simulation Conference
  Proceedings.}\/} (pp. \bibinfo{pages}{901--906}).
\bibitem[{Johnson \& Zhang(2013)}]{Johnacc13}
\bibinfo{author}{Johnson, R.}, \& \bibinfo{author}{Zhang, T.}
  (\bibinfo{year}{2013}).
\newblock \bibinfo{title}{Accelerating stochastic gradient descent using
  predictive variance reduction}.
\newblock In {\it \bibinfo{booktitle}{Advances in Neural Information Processing
  Systems 26}\/} (pp. \bibinfo{pages}{315--323}).
\newblock \bibinfo{publisher}{Curran Associates Inc.}
\bibitem[{Kim et~al.(2008)Kim, Kim \& Kim}]{Kim2008}
\bibinfo{author}{Kim, J.}, \bibinfo{author}{Kim, Y.}, \& \bibinfo{author}{Kim,
  Y.} (\bibinfo{year}{2008}).
\newblock \bibinfo{title}{A {{Gradient}}-{{Based Optimization Algorithm}} for
  {{LASSO}}}.
\newblock {\it \bibinfo{journal}{Journal of Computational and Graphical
  Statistics}\/},  {\it \bibinfo{volume}{17}\/}, \bibinfo{pages}{994--1009}.
\bibitem[{King \& Zeng(2001)}]{King2001}
\bibinfo{author}{King, G.}, \& \bibinfo{author}{Zeng, L.}
  (\bibinfo{year}{2001}).
\newblock \bibinfo{title}{Logistic {{Regression}} in {{Rare Events Data}}}.
\newblock {\it \bibinfo{journal}{Political Analysis}\/},  {\it
  \bibinfo{volume}{9}\/}, \bibinfo{pages}{137--163}.
\bibitem[{Le~Thi et~al.(2008)Le~Thi, Le, Nguyen \& Pham~Dinh}]{LT08}
\bibinfo{author}{Le~Thi, H.~A.}, \bibinfo{author}{Le, H.~M.},
  \bibinfo{author}{Nguyen, V.~V.}, \& \bibinfo{author}{Pham~Dinh, T.}
  (\bibinfo{year}{2008}).
\newblock \bibinfo{title}{A {{DC}} programming approach for feature selection
  in support vector machines learning}.
\newblock {\it \bibinfo{journal}{Advances in Data Analysis and
  Classification}\/},  {\it \bibinfo{volume}{2}\/}, \bibinfo{pages}{259--278}.
\bibitem[{Le~Thi et~al.(2017)Le~Thi, Le, Phan \& Tran}]{LeThi17-ICML}
\bibinfo{author}{Le~Thi, H.~A.}, \bibinfo{author}{Le, H.~M.},
  \bibinfo{author}{Phan, D.~N.}, \& \bibinfo{author}{Tran, B.}
  (\bibinfo{year}{2017}).
\newblock \bibinfo{title}{{Stochastic DCA for the Large-sum of Non-convex
  Functions Problem and its Application to Group Variable Selection in
  Classification}}.
\newblock In {\it \bibinfo{booktitle}{{Proceedings of the 34th International
  Conference on Machine Learning }}\/} (pp. \bibinfo{pages}{3394--3403}).
\newblock volume~\bibinfo{volume}{70}.
\bibitem[{Le~Thi \& Pham~Dinh(2005)}]{LeThi05}
\bibinfo{author}{Le~Thi, H.~A.}, \& \bibinfo{author}{Pham~Dinh, T.}
  (\bibinfo{year}{2005}).
\newblock \bibinfo{title}{The {DC} (difference of convex functions) programming
  and {DCA} revisited with {DC} models of real world nonconvex optimization
  problems}.
\newblock {\it \bibinfo{journal}{Annals of Operations Research}\/},  {\it
  \bibinfo{volume}{133}\/}, \bibinfo{pages}{23--46}.
\bibitem[{Le~Thi \& Pham~Dinh(2018)}]{LeThi18}
\bibinfo{author}{Le~Thi, H.~A.}, \& \bibinfo{author}{Pham~Dinh, T.}
  (\bibinfo{year}{2018}).
\newblock \bibinfo{title}{{DC} programming and {DCA}: thirty years of
  developments}.
\newblock {\it \bibinfo{journal}{Mathematical Programming}\/},  (pp.
  \bibinfo{pages}{1--64}).
\bibitem[{Le~Thi et~al.(2015)Le~Thi, Pham~Dinh, Le \& Vo}]{LT15}
\bibinfo{author}{Le~Thi, H.~A.}, \bibinfo{author}{Pham~Dinh, T.},
  \bibinfo{author}{Le, H.~M.}, \& \bibinfo{author}{Vo, X.~T.}
  (\bibinfo{year}{2015}).
\newblock \bibinfo{title}{{{DC}} approximation approaches for sparse
  optimization}.
\newblock {\it \bibinfo{journal}{Eur. J. Oper. Res.}\/},  {\it
  \bibinfo{volume}{244}\/}, \bibinfo{pages}{26--46}.
\bibitem[{Le~Thi \& Phan(2016)}]{LT16}
\bibinfo{author}{Le~Thi, H.~A.}, \& \bibinfo{author}{Phan, D.~N.}
  (\bibinfo{year}{2016}).
\newblock \bibinfo{title}{{{DC Programming}} and {{DCA}} for {{Sparse Optimal
  Scoring Problem}}}.
\newblock {\it \bibinfo{journal}{Neurocomput.}\/},  {\it
  \bibinfo{volume}{186}\/}, \bibinfo{pages}{170--181}.
\bibitem[{Le~Thi et~al.(2019)Le~Thi, Phan \& Pham}]{LeThi2019}
\bibinfo{author}{Le~Thi, H.~A.}, \bibinfo{author}{Phan, D.~N.}, \&
  \bibinfo{author}{Pham, D.~T.} (\bibinfo{year}{2019}).
\newblock \bibinfo{title}{{DCA} based approaches for bi-level variable
  selection and application for estimating multiple sparse covariance
  matrices}.
\newblock {\it \bibinfo{journal}{Revised version Neurocomputing}\/}, .
\bibitem[{Le~Thi (Home~Page)(2005)}]{LeThi2005}
\bibinfo{author}{Le~Thi (Home~Page), H.~A.} (\bibinfo{year}{2005}).
\newblock \bibinfo{title}{{{DC Programming}} and {{DCA}} - {{Website}} of {{Le
  Thi Hoai An}}}.
\newblock
  \bibinfo{howpublished}{http://www.lita.univ-lorraine.fr/\textasciitilde{}lethi/index.php/en/research/dc-programming-and-dca.html}.
\bibitem[{LeCun et~al.(1998)LeCun, Bottou, Orr \& M{\"u}ller}]{Lecun98}
\bibinfo{author}{LeCun, Y.}, \bibinfo{author}{Bottou, L.},
  \bibinfo{author}{Orr, G.~B.}, \& \bibinfo{author}{M{\"u}ller, K.~R.}
  (\bibinfo{year}{1998}).
\newblock \bibinfo{title}{Efficient backprop}.
\newblock In {\it \bibinfo{booktitle}{Neural Networks: Tricks of the Trade}\/}
  (pp. \bibinfo{pages}{9--50}).
\newblock \bibinfo{address}{Berlin, Heidelberg}: \bibinfo{publisher}{Springer
  Berlin Heidelberg}.
\bibitem[{Liao \& Chin(2007)}]{Liao2007}
\bibinfo{author}{Liao, J.~G.}, \& \bibinfo{author}{Chin, K.-V.}
  (\bibinfo{year}{2007}).
\newblock \bibinfo{title}{Logistic regression for disease classification using
  microarray data: Model selection in a large p and small n case}.
\newblock {\it \bibinfo{journal}{Bioinformatics}\/},  {\it
  \bibinfo{volume}{23}\/}, \bibinfo{pages}{1945--1951}.
\bibitem[{Mairal(2015)}]{Maiinc15}
\bibinfo{author}{Mairal, J.} (\bibinfo{year}{2015}).
\newblock \bibinfo{title}{Incremental majorization-minimization optimization
  with application to large-scale machine learning}.
\newblock {\it \bibinfo{journal}{SIAM Journal on Optimization}\/},  {\it
  \bibinfo{volume}{25}\/}, \bibinfo{pages}{829--855}.
\bibitem[{Neveu(1975)}]{neveu75}
\bibinfo{author}{Neveu, J.} (\bibinfo{year}{1975}).
\newblock {\it \bibinfo{title}{Discrete-Parameter Martingales}\/}
  volume~\bibinfo{volume}{10} of {\it \bibinfo{series}{North-Holland
  Mathematical Library}\/}.
\newblock \bibinfo{publisher}{Elsevier}.
\bibitem[{Parikh \& Boyd(2014)}]{Parikh14}
\bibinfo{author}{Parikh, N.}, \& \bibinfo{author}{Boyd, S.}
  (\bibinfo{year}{2014}).
\newblock \bibinfo{title}{Proximal algorithms}.
\newblock {\it \bibinfo{journal}{Found. Trends Optim.}\/},  {\it
  \bibinfo{volume}{1}\/}, \bibinfo{pages}{127--239}.
\bibitem[{Pham~Dinh \& Le~Thi(1997)}]{PLT97}
\bibinfo{author}{Pham~Dinh, T.}, \& \bibinfo{author}{Le~Thi, H.~A.}
  (\bibinfo{year}{1997}).
\newblock \bibinfo{title}{Convex analysis approach to dc programming:
  {{Theory}}, algorithms and applications}.
\newblock {\it \bibinfo{journal}{Acta Mathematica Vietnamica}\/},  {\it
  \bibinfo{volume}{22}\/}, \bibinfo{pages}{289--355}.
\bibitem[{Pham~Dinh \& Le~Thi(1998)}]{PLT98}
\bibinfo{author}{Pham~Dinh, T.}, \& \bibinfo{author}{Le~Thi, H.~A.}
  (\bibinfo{year}{1998}).
\newblock \bibinfo{title}{A {{D}}. {{C}}. {{Optimization Algorithm}} for
  {{Solving}} the {{Trust}}-{{Region Subproblem}}}.
\newblock {\it \bibinfo{journal}{SIAM Journal of Optimization}\/},  {\it
  \bibinfo{volume}{8}\/}, \bibinfo{pages}{476--505}.
\bibitem[{Pham~Dinh \& Le~Thi(2014)}]{PLT14}
\bibinfo{author}{Pham~Dinh, T.}, \& \bibinfo{author}{Le~Thi, H.~A.}
  (\bibinfo{year}{2014}).
\newblock \bibinfo{title}{Recent advances in {DC} programming and {DCA}}.
\newblock {\it \bibinfo{journal}{Transactions on Computational Collective
  Intelligence}\/},  {\it \bibinfo{volume}{8342}\/}, \bibinfo{pages}{1--37}.
\bibitem[{Pham~Dinh \& Souad(1986)}]{PhamDinh1986}
\bibinfo{author}{Pham~Dinh, T.}, \& \bibinfo{author}{Souad, E.~B.}
  (\bibinfo{year}{1986}).
\newblock \bibinfo{title}{Algorithms for {{Solving}} a {{Class}} of {{Nonconvex
  Optimization Problems}}. {{Methods}} of {{Subgradients}}}.
\newblock In \bibinfo{editor}{J.~B. {Hiriart-Urruty}} (Ed.), {\it
  \bibinfo{booktitle}{North-{{Holland Mathematics Studies}}}\/} (pp.
  \bibinfo{pages}{249--271}).
\newblock \bibinfo{publisher}{{North-Holland}} volume \bibinfo{volume}{129} of
  {\it \bibinfo{series}{Fermat {{Days}} 85: {{Mathematics}} for
  {{Optimization}}}\/}.
\bibitem[{Phan et~al.(2017)Phan, Le~Thi \& Pham~Dinh}]{phan2017}
\bibinfo{author}{Phan, D.~N.}, \bibinfo{author}{Le~Thi, H.~A.}, \&
  \bibinfo{author}{Pham~Dinh, T.} (\bibinfo{year}{2017}).
\newblock \bibinfo{title}{Sparse covariance matrix estimation by {DCA-Based
  Algorithms}}.
\newblock {\it \bibinfo{journal}{Neural Computation}\/},  {\it
  \bibinfo{volume}{29}\/}, \bibinfo{pages}{3040--3077}.
\bibitem[{Phan \& Thi(2019)}]{Phan2019}
\bibinfo{author}{Phan, D.~N.}, \& \bibinfo{author}{Thi, H. A.~L.}
  (\bibinfo{year}{2019}).
\newblock \bibinfo{title}{Group variable selection via $\ell{p,0}$
  regularization and application to optimal scoring}.
\newblock {\it \bibinfo{journal}{Neural Networks}\/}, .
\bibitem[{Reddi et~al.(2016)Reddi, Sra, Poczos \& Smola}]{Reddi2016}
\bibinfo{author}{Reddi, S.~J.}, \bibinfo{author}{Sra, S.},
  \bibinfo{author}{Poczos, B.}, \& \bibinfo{author}{Smola, A.~J.}
  (\bibinfo{year}{2016}).
\newblock \bibinfo{title}{Proximal stochastic methods for {{Nonsmooth Nonconvex
  Finite}}-{{Sum Optimization}}}.
\newblock In {\it \bibinfo{booktitle}{Advances in {{Neural Information
  Processing Systems}}}\/} (pp. \bibinfo{pages}{1145--1153}).
\bibitem[{Robbins \& Monro(1951)}]{Robbins51}
\bibinfo{author}{Robbins, H.}, \& \bibinfo{author}{Monro, S.}
  (\bibinfo{year}{1951}).
\newblock \bibinfo{title}{A stochastic approximation method}.
\newblock {\it \bibinfo{journal}{The Annals of Mathematical Statistics}\/},
  {\it \bibinfo{volume}{22}\/}, \bibinfo{pages}{400--407}.
\bibitem[{Schmidt et~al.(2017)Schmidt, Le~Roux \& Bach}]{Schmidt2017}
\bibinfo{author}{Schmidt, M.}, \bibinfo{author}{Le~Roux, N.}, \&
  \bibinfo{author}{Bach, F.} (\bibinfo{year}{2017}).
\newblock \bibinfo{title}{Minimizing finite sums with the stochastic average
  gradient}.
\newblock {\it \bibinfo{journal}{Mathematical Programming}\/},  {\it
  \bibinfo{volume}{162}\/}, \bibinfo{pages}{83--112}.
\bibitem[{Shalev-Schwartz \& Zhang(2013)}]{Shalev2013}
\bibinfo{author}{Shalev-Schwartz, S.}, \& \bibinfo{author}{Zhang, T.}
  (\bibinfo{year}{2013}).
\newblock \bibinfo{title}{Stochastic dual coordinate ascent methods for
  regularized loss minimization}.
\newblock {\it \bibinfo{journal}{Journal of Machine Learning Research}\/},
  {\it \bibinfo{volume}{14}\/}, \bibinfo{pages}{567--599}.
\bibitem[{Subasi \& Er{\c c}elebi(2005)}]{Subasi2005}
\bibinfo{author}{Subasi, A.}, \& \bibinfo{author}{Er{\c c}elebi, E.}
  (\bibinfo{year}{2005}).
\newblock \bibinfo{title}{Classification of {{EEG}} signals using neural
  network and logistic regression}.
\newblock {\it \bibinfo{journal}{Comput. Methods Programs Biomed.}\/},  {\it
  \bibinfo{volume}{78}\/}, \bibinfo{pages}{87--99}.
\bibitem[{Vincent \& Hansen(2014)}]{Vincent2014}
\bibinfo{author}{Vincent, M.}, \& \bibinfo{author}{Hansen, N.~R.}
  (\bibinfo{year}{2014}).
\newblock \bibinfo{title}{Sparse group lasso and high dimensional multinomial
  classification}.
\newblock {\it \bibinfo{journal}{Comput. Stat. Data Anal.}\/},  {\it
  \bibinfo{volume}{71}\/}, \bibinfo{pages}{771--786}.
\bibitem[{Witten \& Tibshirani(2011)}]{Witten2011}
\bibinfo{author}{Witten, D.~M.}, \& \bibinfo{author}{Tibshirani, R.}
  (\bibinfo{year}{2011}).
\newblock \bibinfo{title}{Penalized classification using {{Fisher}}'s linear
  discriminant}.
\newblock {\it \bibinfo{journal}{Journal of the Royal Statistical Society:
  Series B}\/},  {\it \bibinfo{volume}{73}\/}, \bibinfo{pages}{753--772}.

\end{thebibliography}

\end{document}